\documentclass[final,leqno]{siamltex704}
\usepackage{amsmath}
\usepackage{graphicx}
\usepackage[notcite,notref]{showkeys}
\usepackage{mathrsfs}
\usepackage{graphics}
\usepackage{float}
\usepackage{amsfonts,amssymb}
\usepackage{dsfont}
\usepackage{pifont}
\usepackage{wrapfig} 
\usepackage{hyperref}
\usepackage{multirow}
\usepackage{color}
\numberwithin{equation}{section}
\usepackage{threeparttable}
\def\3bar{{|\hspace{-.02in}|\hspace{-.02in}|}}
\def\E{{\mathcal{E}}}
\def\T{{\mathcal{T}}}

\def\pT{{\partial T}}

\def\W{{\mathcal{W}}}

\def\bn{{\mathbf{n}}}
\def\bb{{\mathbf{b}}}

\newtheorem{remark}{Remark}[section]
\newtheorem{algorithm}{Primal-Dual Weak Galerkin Algorithm}[section]

\title {A New Primal-Dual Weak Galerkin Method for Elliptic Interface Problems with Low Regularity Assumptions}

\begin{document}

\author{
Waixiang Cao \thanks{School of Mathematical Sciences, Beijing Normal University, Beijing 100875, China (caowx@bnu.edu.cn). The research of Waixiang Cao was partially supported by NSFC grant No. 11871106.}
\and
Chunmei Wang \thanks{Department of Mathematics \& Statistics, Texas Tech University, Lubbock, TX 79409, USA (chunmei.wang@ttu.edu). The research of Chunmei Wang was partially supported by National Science Foundation Award DMS-1849483.}
\and
Junping Wang\thanks{Division of Mathematical
Sciences, National Science Foundation, Alexandria, VA 22314
(jwang@nsf.gov). The research of Junping Wang was supported in part by the
NSF IR/D program, while working at National Science Foundation.
However, any opinion, finding, and conclusions or recommendations
expressed in this material are those of the author and do not
necessarily reflect the views of the National Science Foundation.}}

\maketitle

\begin{abstract}
This article introduces a new primal-dual weak Galerkin (PDWG) finite element method for second order elliptic interface problems with ultra-low regularity assumptions on the exact solution and the interface and boundary data. It is proved that the PDWG method is stable and accurate with optimal order of error estimates in discrete and Sobolev norms. In particular, the error estimates are derived under the low regularity assumption of $u\in H^{\delta}(\Omega)$ for $\delta > \frac12$ for the exact solution $u$. Extensive numerical experiments are conducted to provide numerical solutions that verify the efficiency and accuracy of the new PDWG method.
\end{abstract}



\pagestyle{myheadings}

\section{Introduction}
In this paper we are concerned with the development of a new primal-dual weak Galerkin (PDWG) finite element method for second order elliptic interface problems with low regularity assumptions on the exact solution and the interface and boundary data. To this end, let $N$ and $M$  be two positive integers and $\Omega\subset \mathbb R^d$($d = 2, 3$) is an open bounded domain with piecewise smooth Lipschitz boundary $\partial \Omega$. The domain $\Omega$ is partitioned into a set of subdomains $\{\Omega_i\}_{i=1}^N$ with piecewise smooth Lipschitz
boundary $\partial \Omega_i$ for $i=1, \cdots, N$; $\Gamma=\bigcup_{i=1}^N \partial \Omega_i\setminus \partial \Omega$ is the interface between the subdomains in the sense that
$$
\Gamma =\bigcup_{m=1}^{M} \Gamma_m,
$$
where there exist $i, j \in \{1, \cdots, N\}$ such that $\Gamma_m=\partial \Omega_i \cap\partial \Omega_j$ for $m=1, \cdots, M$.
The elliptic interface problem seeks an unknown function $u$ satisfying
\begin{align}\label{model}
-\nabla\cdot (a_i(x) \nabla u_i)+\nabla \cdot (\bb_i(x)u_i)+c_iu_i=&f_i,   &\text{in}\ \Omega_i, i=1,\cdots, N,\\
 u_i=&g_i, & \text{on}\ \partial \Omega_i\cap\partial\Omega, i=1, \cdots, N,\label{bdg}\\
[\![u]\!]_{\Gamma_m}=&\phi_m,   & \text{on}\quad   \Gamma_m, m=1, \cdots, M,\label{bdg1} \\
[\![(a \nabla u -\bb u) \cdot \bn]\!]_{\Gamma_m} =&\psi_m,   &\text{on}\quad   \Gamma_m, m=1, \cdots, M,
\label{model2}
\end{align}
where $u_i=u|_{\Omega_i}$,  $a_i=a|_{\Omega_i},\ \  \bb_i={\bf b}|_{\Omega_i},\ \ c_i=c|_{\Omega_i},$
and $[\![(a \nabla u -\bb u) \cdot \bn]\!]_{\Gamma_m}=(a_i \nabla u_i -\bb_i u_i) \cdot \bn_i+(a_j \nabla u_j -\bb_j u_j) \cdot \bn_j$ with $\bn_i$ and $\bn_j$ being the unit outward normal directions to $\partial \Omega_i\cap \Gamma_m$ and $\partial \Omega_j\cap \Gamma_m$, and $[\![u]\!]_{\Gamma_m}=u_i|_{  \Gamma_m}-u_j|_{\Gamma_m}$.
Assume the elliptic coefficients $a$, $\bb$ and $c$ are piecewise smooth with respect to the partition $\Omega=\cup_{i=1}^N \Omega_i$. We further assume that
$a(x)$ is symmetric and positive definite matrices uniformly in $\Omega$.

Elliptic interface problems arise in many applications of mathematical modeling and simulation of practical problems in science and engineering. These applications include computational electromagnetic \cite{r18, r23, r56, r55}, fluid mechanics \cite{r31}, materials science \cite{r25, r29}, and biological science \cite{r54, r15, r8}, to mention just a few. The physical solution to interface problems often possesses discontinuity and/or non-smoothness across the interfaces so that the standard numerical methods will not work at their full capacity. To address this challenge,  many finite element methods (FEMs) \cite{mwwz, r4, r10, r13, r42} and finite difference methods based on Cartesian grids \cite{r41, r40, r32, r34, gibou} have been developed for effective solving of the elliptic interface problem in the last several decades. In the classical FEMs, unstructured partitions were employed to deal with the irregularity of the domain geometry, particularly around the interface and domain boundary. The interface-fitted FEMs are based on proper formulations of the interface problem combined with finite element partitions that align well with the interface. The penalty methods or Lagrangian multiplier approaches were developed in \cite{r16, r6} by imposing the interface condition in the weak formulation. Incorporating the interface conditions into the numerical formulation has the potential of not only increasing the accuracy of the approximate solutions near the interface, but also the flexibility of allowing the use of computational grids that do not align with the physical interfaces. The discontinuous Galerkin (DG) methods \cite{r12, r17, r21, r30} were developed by using Galerkin projections and properly defined numerical fluxes to enforce the interface conditions in a weak sense. Weak Galerkin (WG) FEMs have been developed in \cite{r37, mwwz} by using discrete weak differential operators in the usual variational form of the elliptic interface problem together with a treatment of the interface condition via the boundary unknowns associated with the weak finite element approximations. The embedded or immersed FEMs \cite{r13, r19, r26, r47, r20, r27, r35, r28, r49, r50} were devised to allow the interface to cut through finite elements for problems with moving interfaces and complex topology. Consequently, structured Cartesian meshes could be used to avoid the time-consuming mesh generation process in the immersed FEMs. Recently, a Hybrid High-Order (HHO) method on unfitted meshes was designed and analyzed in \cite{burman-cutFEM} for elliptic interface problems by means of a consistent penalty method in which the curved interface is allowed to cut through the mesh cells in a general fashion.

Numerical methods in the context of finite differences for the elliptic interface problem include the ghost fluid method \cite{r14}, maximum principle preserving and explicit jump immersed interface method (IIM) \cite{r33, r48, r3, r34}, coupling interface method \cite{r11}, piecewise-polynomial interface method \cite{r9}, and matched interface and boundary (MIB) method \cite{r56, r57, r52}. For problems with non-smooth interfaces, some second order finite difference schemes have been devised in the context of the MIB framework in 2D and 3D \cite{r53, r54, r49, r50}.
Other algorithms based on various mathematical techniques for the elliptic interface problem include the integral equation method \cite{r36, r51}, the finite volume method \cite{r39}, and the virtual node method \cite{r2, r22}.

Even though successes have been achieved in the endeavor of solving elliptic interface problems, challenges remain in the search of new and efficient numerical algorithms for problems with very complicated interface geometries, and for problems with low-regularity solutions. The low-regularity of the solution is often caused by the geometric singularities of the interfaces and/or the non-smoothness of the interface data \cite{r37, r27}.

The goal of this paper is to develop a new numerical method for the elliptic interface problem \eqref{model}-\eqref{model2} which is applicable to solutions with low-regularity assumptions. This new method is devised by coupling a weak formulation of \eqref{model}-\eqref{model2} that is derivative-free on the exact solution $u$ with its dual equation, yielding a new primal-dual weak Galerkin finite element method (PDWG). Compared with the WG methods \cite{r37, mwwz}, the proposed PDWG results in a symmetric positive definite formulation which allows the use of general meshes such as hybrid meshes, polygonal and polyhedral meshes and meshes with hanging nodes. More importantly, the new PDWG method is applicable to the model problem
\eqref{model}-\eqref{model2} with rough boundary and interface data on $g_i$, $\phi_m$, and $\psi_m$. The new method thus provides an efficient numerical algorithm for elliptic interface problems under low regularity assumptions for the exact solution.

The paper is organized as follows. In Section \ref{Section-02}, we derive a weak formulation for the elliptic interface problem \eqref{model}-\eqref{model2} that is derivative-free on the solution variable. In Section \ref{Section:Hessian}, we briefly review the weak differential operators and their discrete analogies. In Section \ref{Section:WGFEM}, we describe the PDWG method for the model problem \eqref{model} based on the weak formulation (\ref{weakform}) and its dual. In Section \ref{Section:ExistenceUniqueness}, we establish the solution existence, uniqueness, and stability. In Section \ref{Section:error-equation}, we provide an error equation for the PDWG solutions. In Section \ref{Section:error-estimates}, we derive some error estimates based on various regularity assumptions on the exact solution. Finally, in Section \ref{Section:NE} we report a couple of numerical results to illustrated and verify our convergence theory.

\section{Preliminaries and Notations}\label{Section-02}
We follow the standard notations for Sobolev spaces
and norms defined on a given open and bounded domain $D\subset \mathbb{R}^d$ with
Lipschitz continuous boundary. As such, $\|\cdot\|_{s,D}$ and
$|\cdot|_{s,D}$ are used to denote the norm and seminorm in the Sobolev space
$H^s(D)$ for any $s\ge 0$. The inner product in $H^s(D)$
is denoted by $(\cdot,\cdot)_{s,D}$ for $s\ge 0$. The space $H^0(D)$ coincides with
$L^2(D)$ (i.e., the space of square integrable functions), for which the norm and the inner product are denoted as $\|\cdot \|_{D}$ and $(\cdot,\cdot)_{D}$. The space $H^{s}(D)$ for $s<0$ is defined as the dual of $H^{-s}_0(D)$ \cite{gr} through the usual $L^2(D)$ pairing. When
$D=\Omega$ or when the domain of integration is clear from the
context, we shall drop the subscript $D$ in the norm and the inner product
notation.

We introduce the following space
$$
V=\left\{v:\ \ v\in H_0^1(\Omega)\cap \prod_{i=1}^N H^{2}(\Omega_i), \ a \nabla v\in H(\text{div}; \Omega)\ \right\}.
$$
For sufficiently smooth boundary and interface data $g_i$, $\phi_m$, and $\psi_m$, the solution of the elliptic interface problem (\ref{model})-(\ref{model2}) satisfies the following weak formulation:
\begin{equation}\label{weakform}
 (u, {\cal L}(v)-\bb \cdot \nabla v+cv) =\phi(v), \qquad \forall v\in V,
\end{equation}
where ${\cal L}(v)|_{\Omega_i}=-\nabla \cdot (a_i(x) \nabla v|_{\Omega_i})$ and
$$
\phi(v)=(f, v)- \sum_{i=1}^N \langle g_i, a_i\nabla v\cdot \bn_i \rangle_{\partial \Omega_i \cap \partial \Omega} -\sum_{m=1}^M \langle \phi_m, a_i\nabla v\cdot \bn_m \rangle_{\Gamma_m}+ \sum_{m=1}^M\langle \psi_m, v \rangle_{\Gamma_m}.
$$
Here $f|_{\Omega_i}=f_i$, $\bn_i$ in the term $\sum_{i=1}^N \langle g_i, a_i\nabla v\cdot \bn_i \rangle_{\partial \Omega_i \cap \partial \Omega} $ is the unit outward normal direction to the boundary $\partial \Omega_i \cap \partial \Omega$ for $i=1, \cdots, N$. The unit vector $\bn_m$ is normal to the interface $\Gamma_m$ and has a direction consistent with the interface condition \eqref{model2}.
The weak form \eqref{weakform} can be derived by testing \eqref{model} against any $v\in V$ followed by twice use of the divergence theorem.
For boundary and interface data that are not smooth (e.g., for $\phi_m, \psi_m\in L^2(\Gamma_m)$ and $g_i\in L^2(\partial\Omega_i\cap\partial\Omega)$), the elliptic interface problem may not possess a strong solution satisfying (\ref{model})-(\ref{model2}) in the classical sense, but it may have a solution with low-regularity that satisfies the weak form \eqref{weakform}.

\begin{definition}
A function $u\in L^2(\Omega)$ is said to be a weak solution of the elliptic interface problem
(\ref{model})-(\ref{model2}) if it satisfies \eqref{weakform}.
\end{definition}

The dual or adjoint problem to \eqref{weakform} seeks an unknown function $\lambda$ such that
\begin{equation}\label{EQ:dualform}
(w, {\cal L}(\lambda)-\bb \cdot \nabla \lambda+c\lambda) =\chi(w), \qquad \forall w\in H^\epsilon(\Omega),
\end{equation}
where $\chi$ is a given functional in $H^\epsilon(\Omega)$.
In the rest of the paper, we assume that the dual or adjoint problem \eqref{EQ:dualform} has one and only one solution in $H^{2-\epsilon}(\Omega)$ with the following regularity estimate
  \begin{equation}\label{EQ:H2Regularity}
  \|\lambda\|_{2-\epsilon}\le C\|\chi\|_{-\epsilon}.
  \end{equation}
This regularity assumption implies that when $\chi\equiv 0$, the dual problem \eqref{EQ:dualform} has only the trivial solution $\lambda\equiv 0$. We point out that the adjoint problem is a regular second order elliptic problem involving no interfaces at all.

The weak variational problem (or primal equation) \eqref{weakform} and its dual form \eqref{EQ:dualform} are seemingly unrelated to each other in the continuous case. However, they are strongly connected and support each other in the context of the weak Galerkin approach for each of them. The rest of the paper will reveal this connection and show how they support each other and jointly provide an efficient numerical method for the elliptic interface problem (\ref{model})-(\ref{model2}).

\section{Weak Differential Operators}\label{Section:Hessian}
The two principal differential operators in the weak formulation (\ref{weakform}) for the second order elliptic interface problem (\ref{model}) are ${\cal L}$ and the gradient operator $\nabla$.  The discrete weak version for ${\cal L}$ and $\nabla$ has been introduced in \cite{wwconvdiff, wy3655}. For completeness, we shall briefly review their definition in this section.

Let $T$ be a polygonal or polyhedral domain with boundary $\partial T$. A weak function on $T$ refers to a triplet $\sigma=\{\sigma_0,\sigma_b, \sigma_n\}$ with $\sigma_0\in L^2(T)$, $\sigma_b\in L^{2}(\partial T)$ and $\sigma_n\in L^{2}(\partial T)$. Here $\sigma_0$ and $\sigma_b$ are used to represent the value of $\sigma$ in the interior and on the boundary of $T$ and $\sigma_n$ is reserved for the value of $a\nabla \sigma \cdot \bn$ on $\partial T$. Note that $\sigma_b$ and $\sigma_n$ may not necessarily be the trace of $\sigma_0$ and $a\nabla \sigma_0  \cdot \bn$ on $\partial T$, respectively. Denote by $\W(T)$ the space of weak functions on $T$:
\begin{equation}\label{2.1}
\W(T)=\{\sigma=\{\sigma_0,\sigma_b, \sigma_n \}: \sigma_0\in L^2(T), \sigma_b\in
L^{2}(\partial T), \sigma_n\in L^{2}(\partial T)\}.
\end{equation}

The weak action of ${\cal L}=-\nabla \cdot (a\nabla)$ on $\sigma\in \W(T)$, denoted by ${\cal L}_w \sigma$, is defined as a linear functional on $H^2(T)$ such that
\begin{equation*}
({\cal L}_w \sigma, w)_T= (\sigma_0, {\cal L} w)_T+\langle \sigma_b, a \nabla w\cdot
\textbf{n}\rangle_{\partial T}- \langle \sigma_n,w\rangle_{\partial T},
\end{equation*}
for all $w \in  H^2(T)$.

The weak gradient of $\sigma\in \W(T)$, denoted by $\nabla_w \sigma$, is defined as a linear functional on $[H^1(T)]^d$ such that
\begin{equation*}
(\nabla_w  \sigma,\boldsymbol{\psi})_T=-(\sigma_0,\nabla \cdot \boldsymbol{\psi})_T+\langle \sigma_b,\boldsymbol{\psi}\cdot \textbf{n}\rangle_{\partial T},
\end{equation*}
for all $\boldsymbol{\psi}\in [H^1(T)]^d$.

Denote by $P_r(T)$ the space of polynomials on $T$ with degree no more than $r$. A discrete version of ${\cal L}_w  \sigma$ for $\sigma\in \W(T)$, denoted by ${\cal L}_{w, r, T} \sigma$, is defined as the unique polynomial in $P_r(T)$ satisfying
\begin{equation}\label{Loperator1-1}
({\cal L}_{w, r, T} \sigma, w)_T= (\sigma_0, {\cal L} w)_T+\langle \sigma_b, a \nabla w\cdot
\textbf{n}\rangle_{\partial T}- \langle \sigma_n, w\rangle_{\partial T}, \quad\forall w \in P_r(T),
\end{equation}
which, from the usual integration by parts, gives
\begin{equation}\label{Loperator1-2}
({\cal L}_{w, r, T} \sigma, w)_T= ({\cal L}\sigma_0, w)_T-\langle \sigma_0-\sigma_b, a \nabla w\cdot
\textbf{n}\rangle_{\partial T}+ \langle a\nabla \sigma_0\cdot \bn-\sigma_n, w \rangle_{\partial T},
\end{equation}
for all $w \in P_r(T)$, provided that $\sigma_0\in H^2(T)$.

A discrete version of $\nabla_{w}\sigma$  for $\sigma\in \W(T)$, denoted by $\nabla_{w, r, T}\sigma$, is defined as a unique polynomial vector in $[P_r(T) ]^d$ satisfying
\begin{equation}\label{disgradient}
(\nabla_{w, r, T} \sigma, \boldsymbol{\psi})_T=-(\sigma_0, \nabla \cdot \boldsymbol{\psi})_T+\langle \sigma_b, \boldsymbol{\psi} \cdot \textbf{n}\rangle_{\partial T}, \quad\forall\boldsymbol{\psi}\in [P_r(T)]^d,
\end{equation}
which, from the usual integration by parts, gives
\begin{equation}\label{disgradient*}
(\nabla_{w, r, T} \sigma, \boldsymbol{\psi})_T= (\nabla \sigma_0, \boldsymbol{\psi})_T-\langle \sigma_0- \sigma_b, \boldsymbol{\psi} \cdot \textbf{n}\rangle_{\partial T}, \quad\forall\boldsymbol{\psi}\in [P_r(T)]^d,
\end{equation}
provided that $\sigma_0\in H^1(T)$.

\section{Primal-Dual Weak Galerkin Algorithm}\label{Section:WGFEM}
Let ${\cal T}_h$ be a finite element partition of the domain $\Omega$ consisting of
polygons or polyhedra that are shape-regular \cite{wy3655}. Assume that the edges/faces of the elements in ${\cal T}_h$ align with the interface $\Gamma$.
The partition ${\cal T}_h$ can be grouped into $N$ sets of elements denoted by ${\cal T}_h^{i}={\cal T}_h \cap \Omega_i$, so that each ${\cal T}_h^i$ provides a finite element partition for the subdomain $\Omega_i$ for $i=1, \cdots, N$.
The intersection of the partition ${\cal T}_h$ also introduces a finite element partition for the interface $\Gamma$, denoted by $\Gamma_h$.
Denote by ${\mathcal E}_h$ the
set of all edges or flat faces in ${\cal T}_h$ and  ${\mathcal
E}_h^0={\mathcal E}_h \setminus
\partial\Omega$ the set of all interior edges or flat faces.
Denote by $h_T$ the meshsize of $T\in {\cal T}_h$ and
$h=\max_{T\in {\cal T}_h}h_T$ the meshsize for the partition
${\cal T}_h$.

For any given integer $k\geq 1$, denote by
$W_k(T)$ the local discrete space of  the weak functions given by
$$
W_k(T)=\{\{\sigma_0,\sigma_b, \sigma_n\}:\sigma_0\in P_k(T),\sigma_b\in
P_k(e), \sigma_n \in
P_{k-1}(e),e\subset \partial T\}.
$$
Patching $W_k(T)$ over all the elements $T\in {\cal T}_h$
through a common value $\sigma_b$ on the interior interface $\E_h^0$, we arrive at the following
weak finite element space $W_h$:
$$
W_h=\big\{\{\sigma_0, \sigma_b, \sigma_n\}:\{\sigma_0, \sigma_b,  \sigma_n\}|_T\in W_k(T), \forall T\in {\cal T}_h \big\}.
$$
Note that $\sigma_n$ has two values $\sigma_n^L$ and $\sigma_n^R$ satisfying $\sigma_n^L+\sigma_n^R=0$ on each interior interface $e=\partial T_L\cap \partial T_R \in {\cal E}_h^0$ as seen from the two elements $T_L$ and $T_R$. Denote by $W_h^0$ the subspace of $W_h$ with homogeneous boundary values; i.e.,
\begin{equation*}\label{wh0}
 W_h^0=\{\{\sigma_0,\sigma_b, \sigma_n\} \in W_h:\  \sigma_b|_{\partial\Omega}=0 \}.
\end{equation*}
Denote by $M_h$ the finite element space consisting of piecewise polynomials of degree $k-1$; i.e.,
\begin{equation*}\label{mh}
M_h=\{w: w|_T\in P_{k-1}(T),  \forall T\in {\cal T}_h\}.
\end{equation*}

For simplicity of notation and without confusion, for any $\sigma\in
W_h$, denote by ${\cal L}_{w} \sigma$ and $\nabla_{w}\sigma$  the discrete weak actions ${\cal L}_{w, k-1, T} \sigma$ and $\nabla_{w, k-1, T}\sigma$ computed by using (\ref{Loperator1-1}) and (\ref{disgradient}) on each element $T$; i.e.,
 $$
({\cal L}_{w} \sigma)|_T={\cal L}_{w, k-1, T}(\sigma|_T), \qquad \sigma\in W_h,
$$
$$
(\nabla_{w}\sigma)|_T= \nabla_{w, k-1, T}(\sigma|_T), \qquad \sigma\in W_h.
$$

For any $\sigma, \lambda\in W_h$ and $u\in M_h$, we introduce the
following bilinear forms
\begin{align} \label{EQ:local-stabilizer}
s(\sigma, \lambda)=&\sum_{T\in {\cal T}_h}s_T(\sigma, \lambda),
\\
b(u, \lambda)=&\sum_{T\in {\cal T}_h}b_T(u, \lambda),  \label{EQ:local-bterm}
\end{align}
where
\begin{equation*}
\begin{split}
s_T(\sigma, \lambda)=&h_T^{-3}\langle \sigma_0-\sigma_b, \lambda_0-\lambda_b\rangle_{\partial T}+  h_T^{-1}\langle a \nabla \sigma_0 \cdot \bn-\sigma_n, a \nabla \lambda_0 \cdot \bn-\lambda_n\rangle_{\partial T}\\
&+\tau (\sigma_0, \lambda_0)_T,\\
b_T(u, \lambda)=&(u, {\cal L}_w \lambda-\bb \cdot \nabla_w \lambda+c\lambda_0)_T,
\end{split}
\end{equation*}
with $\tau >0$ being a parameter.

The following is the primal-dual weak Galerkin scheme for the second order elliptic interface problem (\ref{model}) based on the variational formulation (\ref{weakform}).
\begin{algorithm}
Find $(u_h;\lambda_h)\in M_h \times W_{h}^0$, such that
\begin{eqnarray}\label{32}
s(\lambda_h, \sigma)+b(u_h, \sigma)&=& \phi_h(\sigma), \qquad \forall \sigma\in W_{h}^0,\\
b(w, \lambda_h)&=&0,\qquad \quad \qquad \forall w\in M_h.\label{2}
\end{eqnarray}
\end{algorithm}
Here $\phi_h(\sigma)=(f, \sigma_0)-\sum_{i=1}^N \langle g_i, \sigma_n \rangle_{\partial \Omega_i \cap \partial \Omega}-\sum_{m=1}^M \langle \phi_m, \sigma_n  \rangle_{\Gamma_m}+\sum_{m=1}^M\langle \psi_m, \sigma_b  \rangle_{\Gamma_m}$.

\section{Stability Analysis}\label{Section:ExistenceUniqueness}
Let $Q_0$ be the $L^2$ projection operator onto $P_k(T)$, $k\geq 1$. Analogously, denote by $Q_b$ and $Q_n$ the $L^2$ projection operators onto $P_{k}(e)$ and $P_{k-1}(e)$, respectively, for $e\subset\partial T$. For $w\in H^1(\Omega)$, define the $L^2$ projection $Q_h w\in W_h$ as follows
$$
Q_hw|_T=\{Q_0w,Q_bw, Q_n( a\nabla w \cdot \bn)\}.
$$
The $L^2$ projection operator onto the finite element space $M_h$ is denoted as ${\cal Q}_h^{k-1}$.

For simplicity, we assume the coefficient coefficients $a(x)$, $\bb(x)$ and $c(x)$ are piecewise constants with respect to the finite element partition ${\cal T}_h$. The analysis can be extended to piecewise smooth coefficients $a(x)$, $\bb(x)$ and $c(x)$ without any difficulty.

\begin{lemma}\label{Lemma5.1} \cite{wwconvdiff, wy3655} The operators $Q_h$ and ${\cal Q}^{k-1}_h$ satisfy the following commutative properties:
\begin{eqnarray}\label{l}
{\cal L}_{w}(Q_h w) &=& {\cal Q}_h^{k-1}( {\cal L} w), \qquad  \forall  w\in H^2(T),\\
\nabla_{w}(Q_h w) &=& {\cal Q}^{k-1}_h(\nabla w),  \qquad  \forall  w\in H^1(T). \label{l-2}
\end{eqnarray}
\end{lemma}

The stabilizer $s(\cdot, \cdot)$ given in (\ref{EQ:local-stabilizer}) naturally induces the following semi-norm in the weak finite element space $W_{h}$
\begin{equation}\label{norm-new}
\3bar \rho  \3bar_w=
s(\rho, \rho)^{\frac{1}{2}},\qquad  \rho\in W_{h}.
\end{equation}

\begin{lemma}\label{lemmanorm}
The semi-norm $\3bar \cdot \3bar_w$ given in (\ref{norm-new}) defines a norm in the linear space $W_{h}^0$.
\end{lemma}

\begin{proof}
It suffices to verify the positivity property for $\3bar\cdot\3bar_w$. Assume $\3bar \rho \3bar_w=0$ for some $\rho\in W_{h}^0$. It follows that $\rho_0=0$ on each element $T$, $\rho_0=\rho_b$ and $a\nabla \rho_0 \cdot \bn =\rho_n$ on each $\partial T$. We thus obtain $\rho_0\in C^0(\Omega )$ and further $\rho
_0\equiv 0$ in $\Omega$. Using $\rho_0=\rho_b$ and $a\nabla \rho_0 \cdot \bn =\rho_n$ on each $\partial T$ gives $\rho
_b\equiv 0$, $\rho
_n\equiv 0$, and further $\rho \equiv 0$ in $\Omega$.  This completes the proof of the lemma.
\end{proof}

Consider the auxiliary problem of seeking $\Phi$ such that
\begin{eqnarray}\label{pro-new}
{\cal L}(\Phi)- \bb \cdot\nabla \Phi+c\Phi & = &\psi, \qquad \text{in}\ \Omega,\\
\Phi & = & 0,  \qquad\text{on}\  \partial\Omega,\label{pro0-new}
\end{eqnarray}
where $\psi\in L^2(\Omega)$ is a given function. Assume that the problem (\ref{pro-new})-(\ref{pro0-new}) has a solution $\Phi\in \prod_{i=1}^N H^{1+\gamma}(\Omega_i)$ satisfying
\begin{equation}\label{regu2}
\left(\sum_{i=1}^N \|\Phi\|^2_{1+\gamma, \Omega_i}\right)^{\frac{1}{2}}\leq C\|\psi\|_{\gamma-1},
\end{equation}
with some parameter $\gamma\in (\frac{1}{2}, 1]$.

\begin{lemma}\label{lem3-new} ({\it inf-sup} condition)  Under the assumption of \eqref{regu2},
there exists a constant $\beta>0$ independent of the meshsize $h$ such that
\begin{equation}\label{EQ:inf-sup-condition-02}
\sup_{0\neq \sigma \in W_h^0} \frac{b(v, \sigma)}{\3bar \sigma \3bar_w}\geq \beta h^{1-\gamma}\|v\|_{1-\gamma}, \qquad \forall v\in M_h.
\end{equation}
\end{lemma}

\begin{proof} Let $\Phi$ be the solution of (\ref{pro-new})-(\ref{pro0-new}) satisfying \eqref{regu2}. By letting $\sigma=Q_h\Phi$, we have from Lemma \ref{Lemma5.1} and (\ref{pro-new}) that
\begin{equation}\label{EQ:April05:100}
\begin{split}
b(v,\sigma) = & \sum_{T\in {\cal T}_h}(v, {\cal L}_w (Q_h\Phi)-\bb \cdot \nabla_w Q_h\Phi+cQ_0\Phi)_T \\
= & \sum_{T\in {\cal T}_h}(v, {\cal Q}^{k-1}_h({\cal L}\Phi)-\bb \cdot {\cal Q}_h^{k-1} (\nabla \Phi)+cQ_0\Phi)_T \\
= & \sum_{T\in {\cal T}_h}(v, {\cal L}\Phi-\bb \cdot \nabla \Phi+c\Phi)_T
= \sum_{T\in {\cal T}_h}(v, \psi)_T
\end{split}
\end{equation}
for all $v\in M_h$. From the trace inequality (\ref{trace-inequality}), the estimate \eqref{3.1} with $m=\gamma$, and the estimate (\ref{regu2}), there holds
\begin{equation}\label{EQ:Estimate:002}
\begin{split}
 \sum_{T\in {\cal T}_h }h_T^{-3}\int_{\partial
T}|\sigma_0-\sigma_b|^2ds=& \sum_{T\in {\cal T}_h}
h_T^{-3}\int_{\partial T}|Q_0\Phi-Q_b \Phi|^2ds\\
\leq &\sum_{T\in {\cal T}_h} h_T^{-3}\int_{\partial T}|Q_0\Phi-\Phi|^2ds\\
\leq & \sum_{T\in {\cal T}_h} C\{h_T^{-4}\|Q_0\Phi-\Phi\|^2_T
+h_T^{2\gamma-4}\| Q_0\Phi-  \Phi\|^2_{\gamma, T}\}\\
\leq &Ch^{2\gamma-2}\sum_{i=1}^N \|\Phi\|^2_{1+\gamma, \Omega_i}
\leq Ch^{2\gamma-2}\|\psi\|^2_{\gamma-1}.
 \end{split}
\end{equation}
A similar analysis can be applied to yield the following estimate:
\begin{equation}\label{EQ:Estimate:003}
\begin{split}
\sum_{T\in {\cal T}_h }h_T^{-1}\int_{\partial T}|a \nabla
\sigma_0 \cdot \bn-\boldsymbol{\sigma}_n|^2ds \leq & Ch^{2\gamma-2}\|\psi\|^2_{\gamma-1}.
 \end{split}
\end{equation}
Furthermore, by letting $\sigma=Q_h\Phi$ and using \eqref{regu2}, we arrive at
\begin{equation}\label{EQ:Estimate:004}
\begin{split}
\sum_{T\in {\cal T}_h }\tau\int_{T} |\sigma_0|^2dT =& \sum_{T\in {\cal T}_h }\tau\int_{T} |Q_0\Phi|^2dT\\
\leq & C\sum_{i=1}^N \|\Phi\|^2_{1+\gamma, \Omega_i}
\leq  C \|\psi\|^2_{\gamma-1}. \end{split}
\end{equation}
Combining the estimates (\ref{EQ:Estimate:002})-(\ref{EQ:Estimate:004}) and then using the definition of $\3bar \sigma \3bar_w$, we obtain
\begin{equation}\label{barv}
\3bar \sigma\3bar_w^2  \leq  Ch^{2\gamma-2}\|\psi\|^2_{\gamma-1}.
\end{equation}
Thus, it follows from \eqref{EQ:April05:100} and \eqref{barv} that
\begin{equation*}
\begin{split}
\sup_{0\neq \sigma \in W_h^0} \frac{b(v, \sigma)}{\3bar \sigma \3bar_w}\geq & \sup_{0\neq \sigma=Q_h\Phi \in W_h^0} \frac{b(v, \sigma)}{\3bar \sigma \3bar_w}\\
\geq &\sup_{0\neq \psi \in L^2(\Omega)} \frac{\sum_{T\in {\cal T}_h}(v, \psi)_T}{\3bar Q_h\Phi \3bar_w}\\
\geq &\sup_{0\neq \psi \in L^2(\Omega)} \frac{(v, \psi)}{Ch^{\gamma-1}\|\psi\|_{\gamma-1}}\\
\geq &  \beta  h^{1-\gamma}\|v\|_{1-\gamma},
\end{split}
\end{equation*}
for some constant $\beta$ independent of the meshsize $h$. This completes the proof of the lemma.
\end{proof}

We are now in a position to state the main result on the solution existence and uniqueness for the primal-dual weak Galerkin finite element method (\ref{32})-(\ref{2}).

\begin{theorem}
The primal-dual weak Galerkin finite element scheme (\ref{32})-(\ref{2}) has one and only one solution.
\end{theorem}
\begin{proof} It is sufficient to show that the homogeneous case of (\ref{32})-(\ref{2}) has only the trivial solution. To this end, assume $f=0$, $g_i=0$ for $i=1,\cdots, N$, $\phi_m=0$ and $\psi_m=0$ for $m=1, \cdots, M$ in (\ref{32})-(\ref{2}). This implies $\phi_h(\sigma)=0$ for all $\sigma\in W_h^0$. By choosing $\sigma=\lambda_h$ and $w=u_h$ in (\ref{32})-(\ref{2}), we have
$$
s(\lambda_h, \lambda_h)=0,
$$
which gives $\lambda_h\equiv 0$ from Lemma \ref{lemmanorm}. The equation (\ref{32}) can then be rewritten as
\begin{equation}\label{rew}
b(u_h, \sigma) = 0,\qquad \forall \sigma\in W_{h}^0,
\end{equation}
which, together with Lemma \ref{lem3-new}, implies
\begin{equation*}
0=\sup_{0\neq \sigma \in W_h^0} \frac{b(u_h, \sigma) }{\3bar \sigma \3bar_w}\geq \beta h^{1-\gamma}\|u_h\|_{1-\gamma},
\end{equation*}
so that $u_h\equiv 0$ in $\Omega$. This completes the proof of the theorem.
\end{proof}

\section{Error Equations}\label{Section:error-equation}
The goal of this section is to derive some error equations for the primal-dual weak Galerkin method (\ref{32})-(\ref{2}). The error equations shall play a critical role in the forthcoming convergence analysis.

Let $u$ and $(u_h;\lambda_h) \in M_h\times W_h^0$ be the exact solution of the elliptic interface problem \eqref{model}-\eqref{model2} and its numerical solution arising from the PDWG scheme (\ref{32})-(\ref{2}). Note that the exact Lagrangian multiplier $\lambda=0$ is trivial. Denote the error functions by
\begin{align}\label{error}
e_h&= {\cal Q}^{k-1}_hu-u_h,\ \ e_u={\cal Q}^{k-1}_hu-u,
\\
\epsilon_h&=  Q_h\lambda-\lambda_h=-\lambda_h.\label{error-2}
\end{align}

\begin{lemma}\label{errorequa}
Let $u$ and $(u_h;\lambda_h) \in M_h\times W_h^0$ be the exact solution of elliptic interface problem \eqref{model}-\eqref{model2} and its numerical solution arising from the PDWG scheme (\ref{32})-(\ref{2}). The error functions $e_h$ and $\epsilon_h$ defined in (\ref{error})-(\ref{error-2}) satisfy the following equations:
\begin{eqnarray}\label{sehv}
s(\epsilon_h, \sigma)+b(e_h, \sigma)&=& \ell_u(\sigma), \qquad \forall \sigma\in W_{h}^0,\\
b(w, \epsilon_h)&=&0,\qquad\qquad \forall w\in M_h. \label{sehv2}
\end{eqnarray}
\end{lemma}
Here
 \begin{equation}\label{lu}
\qquad \ell_u(\sigma)=\sum_{T\in {\cal T}_h} \langle  e_u, a\nabla \sigma_0 \cdot \bn - \sigma_n \rangle_{\partial T}+ (\sigma_0, ce_u)_T
 +\langle  \sigma_b-\sigma_0, (a \nabla e_u -\bb e_u)\cdot \bn \rangle_{\partial T}.
\end{equation}

\begin{proof}
Recalling the definition of $b(\cdot,\cdot)$ in \eqref{EQ:local-bterm} and choosing $w={\cal Q}^{k-1}_hu$ in (\ref{Loperator1-1}) and $ \boldsymbol{\psi} =\bb {\cal Q}^{k-1}_hu$ in \eqref{disgradient}, we get
\begin{eqnarray*}
b({\cal Q}^{k-1}_hu, \sigma)
&=&\sum_{T\in {\cal T}_h}({\cal Q}^{k-1}_hu, {\cal L}_w\sigma-\bb \cdot \nabla_w \sigma+c\sigma_0)_T\\
&=&\sum_{T\in {\cal T}_h}(\sigma_0, {\cal L}({\cal Q}^{k-1}_hu))_T+ \langle \sigma_b, a \nabla{\cal Q}^{k-1}_hu \cdot \bn \rangle_{\partial T}-\langle \sigma_n, {\cal Q}^{k-1}_hu\rangle_{\partial T} \\
&+&(\sigma_0, \nabla \cdot(\bb {\cal Q}^{k-1}_hu))_T-\langle \sigma_b, \bb {\cal Q}^{k-1}_hu \cdot \bn\rangle_{\partial T}+({\cal Q}^{k-1}_hu, c\sigma_0)_T.
\end{eqnarray*}
Observe that the exact Lagrangian multiplier $\lambda=0$ is trivial, so that $s(Q_h\lambda, \sigma)=0$ for all $\sigma\in W_h^0$. It follows from \eqref{bdg1} -\eqref{model2} and  $\sigma_b|_{\partial\Omega}=0$ that
\begin{equation}\label{sterm}
\begin{split}
&s(Q_h\lambda, \sigma)+b({\cal Q}^{k-1}_hu, \sigma)\\
=&\sum_{T\in {\cal T}_h}(\sigma_0, {\cal L}({\cal Q}^{k-1}_hu)+\nabla \cdot (\bb {\cal Q}^{k-1}_h u)+c{\cal Q}^{k-1}_h u)_T\\
&+\sum_{T\in {\cal T}_h}( \langle  \sigma_b, (a \nabla({\cal Q}^{k-1}_hu-u) -\bb ({\cal Q}^{k-1}_hu-u))\cdot \bn \rangle_{\partial T}-\langle \sigma_n, {\cal Q}^{k-1}_hu-u \rangle_{\partial T} )\\
&-\sum_{i=1}^N \langle g_i, \sigma_n \rangle_{\partial \Omega_i \cap \partial \Omega}-\sum_{m=1}^M \langle \phi_m, \sigma_n  \rangle_{\Gamma_m}+\sum_{m=1}^M\langle \psi_m, \sigma_b  \rangle_{\Gamma_m}.
\end{split}
\end{equation}
Denote
\[
   I=\sum_{T\in {\cal T}_h}(\sigma_0, {\cal L}({\cal Q}^{k-1}_hu)+\nabla \cdot (\bb {\cal Q}^{k-1}_h u)+c{\cal Q}^{k-1}_h u)_T.
\]
By the usual integration by parts, and (\ref{model})-\eqref{model2}, we have
\begin{eqnarray*}
I&=& \sum_{T\in {\cal T}_h}(\sigma_0, {\cal L}e_u+\nabla \cdot (\bb e_u)+c e_u)_T+(f, \sigma_0)\\
&=& \sum_{T\in {\cal T}_h}\left((a\nabla e_u- \bb e_u, \nabla \sigma_0)_T-\langle (a\nabla e_u -\bb e_u)\cdot \bn, \sigma_0\rangle_{\partial T} \right)+ (\sigma_0, c e_u+f)\\
&=& \sum_{T\in {\cal T}_h}(e_u, {\cal L} \sigma_0)_T+\langle  e_u, a\nabla \sigma_0 \cdot \bn \rangle_{\partial T}- (\bb e_u, \nabla \sigma_0)_T-\langle (a\nabla e_u -\bb e_u)\cdot \bn, \sigma_0\rangle_{\partial T}\\
&&\ + (\sigma_0, c e_u+f)\\
&=& \sum_{T\in {\cal T}_h}\langle  e_u, a\nabla \sigma_0 \cdot \bn \rangle_{\partial T}-\langle (a\nabla e_u -\bb e_u)\cdot \bn, \sigma_0\rangle_{\partial T}+ (\sigma_0, c e_u+f),
\end{eqnarray*}
where we have used $(e_u, \nabla \cdot(a\nabla \sigma_0))_T=0$ and $(\nabla \sigma_0, \bb e_u)_T=0$ due to the orthogonality property of the $L^2$ projection operator ${\cal Q}^{k-1}_h$.

Substituting the above equation of $I$ into (\ref{sterm}) yields
\begin{equation}\label{errterm1}
\begin{split}
& s(Q_h\lambda, \sigma)+b({\cal Q}^{k-1}_hu, \sigma)\\
= & \sum_{T\in {\cal T}_h} \langle  e_u, a\nabla \sigma_0 \cdot \bn - \sigma_n \rangle_{\partial T}+ \langle  \sigma_b-\sigma_0, (a \nabla e_u -\bb e_u)\cdot \bn \rangle_{\partial T}+ (\sigma_0,ce_u+f)_T\\
&  -\sum_{i=1}^N \langle g_i, \sigma_n \rangle_{\partial \Omega_i \cap \partial \Omega}-\sum_{m=1}^M \langle \phi_m, \sigma_n  \rangle_{\Gamma_m}+\sum_{m=1}^M\langle \psi_m, \sigma_b  \rangle_{\Gamma_m}.
\end{split}
\end{equation}
The difference of (\ref{errterm1}) and (\ref{32}) gives (\ref{sehv}).

From $\lambda=0$, the equation (\ref{2}) gives rise to
\begin{equation*}
\begin{split}
b(w, \epsilon_h)=b(w, Q_h \lambda-\lambda_h)=b(w, -\lambda_h) =0,\qquad \forall w\in M_h,
\end{split}
\end{equation*}
which completes the derivation of (\ref{sehv2}). The proof is thus completed.
\end{proof}

The equations (\ref{sehv})-(\ref{sehv2}) are called {\em error
equations} for the primal-dual WG finite element scheme
(\ref{32})-(\ref{2}).

\begin{remark}
For $C^0$-WG elements (i.e.,  $\sigma_0=\sigma_b$ on the boundary of each element), the second term in (\ref{lu}) vanishes so that $l_u(\sigma)$ can be simplified as
 \begin{equation}\label{lu2}
\qquad \ell_u(\sigma)=\sum_{T\in {\cal T}_h} \langle  e_u, a\nabla \sigma_0 \cdot \bn - \sigma_n \rangle_{\partial T}+ (\sigma_0, ce_u)_T,
\end{equation}
which involves no derivative for the exact solution $u$. The expression \eqref{lu2} permits an error estimate under ultra-low regularity assumptions for the solution of the interface problem \eqref{model}-\eqref{model2}.
\end{remark}

\section{Error Estimates}\label{Section:error-estimates}
As $\T_h$ is a shape-regular finite element partition of
the domain $\Omega$, for any $T\in\T_h$ and $\varphi\in H^{\gamma}(T)\ (\frac{1}{2}< \gamma \leq 1)$, the following trace inequality holds true \cite{wy3655}:
\begin{equation}\label{trace-inequality}
\|\varphi\|_{\pT}^2 \leq C
(h_T^{-1}\|\varphi\|_{T}^2+h_T^{2 \gamma-1} \| \varphi\|_{\gamma, T}^2).
\end{equation}

\begin{lemma}\label{Lemma5.2}\cite{wy3655}  Let ${\cal T}_h$ be a
finite element partition of $\Omega$ satisfying the shape regularity
assumptions as specified in \cite{wy3655}.
The following estimates hold true
\begin{eqnarray}\label{3.2}
\sum_{T\in {\cal T}_h}h_T^{2l}\|u-{\cal Q}^{k-1}_hu\|^2_{l,T} &\leq &
Ch^{2m}\|u\|_{m}^2,\\
\sum_{T\in {\cal T}_h} h_T^{2l}\|u- Q_0 u\|^2_{l,T} &\leq& Ch^{2(m+1)}\|u\|_{m+1}^2, \label{3.1}
\end{eqnarray}
where $0\leq l \leq 2$ and $0\leq m\leq k$.
\end{lemma}

The main convergence result can be stated as follows.
\begin{theorem} \label{theoestimate}
Assume $k\geq 1$. Let $u$ and $(u_h;\lambda_h) \in M_h\times W_h^0$ be the exact solution of the elliptic interface problem \eqref{model}-\eqref{model2} and its numerical solution arising from the PDWG scheme (\ref{32})-(\ref{2}). Assume that $u$ is sufficiently regular such that $u\in \prod_{i=1}^N  H^{k}(\Omega_i)\cap H^{1+\gamma}(\Omega_i)$. The following error estimate holds true:
 \begin{equation}\label{erres}
\3bar \epsilon_h \3bar_w+h^{1-\gamma}\|e_h\|_{1-\gamma}\leq Ch^{k}\left((1+\tau^{-1})\sum_{i=1}^N \|u\|^2_{k, \Omega_i}+ \delta _{k, 1}h^{2\gamma}\|u\|^2_{1+\gamma, \Omega_i}\right)^{\frac{1}{2}},\end{equation}
where $\frac{1}{2}< \gamma \leq 1$ and $\delta_{i, j}$ is the Kronecker delta with value 1 when $i=j$ and 0 otherwise. As a result, one has the following optimal order error estimate in the $H^{1-\gamma}$-norm for $u_h$
 \begin{equation}\label{erres:L2}
 \|u - u_h\|_{1-\gamma} \leq Ch^{k+\gamma-1}\left((1+\tau^{-1})\sum_{i=1}^N \|u\|^2_{k, \Omega_i}+ \delta _{k, 1}h^{2\gamma}\|u\|^2_{1+\gamma, \Omega_i}\right)^{\frac{1}{2}}.
 \end{equation}
\end{theorem}

\begin{proof}  Note that $\epsilon_h=-\lambda_h\in W_h^0$.  Choosing  $\sigma = \epsilon_h$  and $w=e_h$ in (\ref{sehv}) and (\ref{sehv2}) gives rise to
\begin{equation}\label{EQ:April7:001}
s(\epsilon_h, \epsilon_h) = \ell_u(\epsilon_h).
\end{equation}

For $k \geq 2$, we use (\ref{EQ:April7:001}), the Cauchy-Schwarz inequality, the equation (\ref{lu}), the trace inequality (\ref{trace-inequality}), and the estimates (\ref{3.2}) with $m=k$  to obtain
\begin{eqnarray*}
&&\3bar \epsilon_h \3bar_w^2 =s(\epsilon_h, \epsilon_h) = |\ell_u(\epsilon_h)| \\
&\leq & \Big(\sum_{T\in {\cal
T}_h}h_T^{-3}\|\epsilon_0-\epsilon_b\|^2_{\partial
T}\Big)^{\frac{1}{2}} \Big(\sum_{T\in {\cal T}_h}h_T^{ 3}\| (a\nabla e_u-\bb e_u)\cdot
\bn\|^2_{\partial T}\Big)^{\frac{1}{2}}\\
 &\ \ + & \Big(\sum_{T\in {\cal T}_h} h_T^{-1}\| a\nabla \epsilon_0\cdot \bn-\epsilon_n\|^2_{\partial T}\Big)^{\frac{1}{2}}
\Big(\sum_{T\in {\cal T}_h} h_T \|e_u
\|^2_{\partial T}\Big)^{\frac{1}{2}}\\
&&\ \ +\Big(\sum_{T\in {\cal T}_h}\tau \|\epsilon_0\|^2_T \Big)^{\frac{1}{2}}
\Big(\sum_{T\in {\cal T}_h} \tau^{-1}\|ce_u\|_T^2 \Big)^{\frac{1}{2}}\\
&\leq& C\3bar \epsilon_h \3bar_w   \left(\sum_{T\in {\cal T}_h}h_T^{2}\| a\nabla e_u-\bb e_u\|^2_{T}
+ h_T^{2\gamma+2} \| a\nabla e_u-\bb e_u\|^2_{\gamma, T}\right. \\
&&\ \ \left. +\|e_u
\|^2_{T}+h_T^{2\gamma} \|e_u\|^{2}_{\gamma, T}+ \tau^{-1}\|ce_u\|_T^2\right)^{\frac{1}{2}}.
\end{eqnarray*}
 Therefore, for $k\ge 2$, we have
\[
  \3bar \epsilon_h \3bar_w^2 \leq C h^{k}\left((1+\tau^{-1})\sum_{i=1}^N\|u\|^2_{k, \Omega_i}\right)^{\frac{1}{2}} \3bar \epsilon_h \3bar_w.
\]
As to the case $k=1$, we have $\nabla  {\cal Q}^{k-1}_hu=0$ and thus
\begin{eqnarray*}
  \3bar \epsilon_h \3bar_w^2 &\leq & C\Big(\sum_{T\in {\cal T}_h}h_T^{2}\| a\nabla u-\bb e_u\|^2_{T}
+ h_T^{2\gamma+2} \| a\nabla u-\bb e_u\|^2_{\gamma, T}\\
&&\ +\|e_u
\|^2_{T}+h_T^{2\gamma} \|e_u\|^{2}_{\gamma, T}+ \tau^{-1}\|ce_u\|_T^2\Big)^{\frac{1}{2}},\\
  &\leq &  C h\left((1+\tau^{-1})\sum_{i=1}^N \|u\|^2_{1, \Omega_i}+ h^{2\gamma}\|u\|^2_{1+\gamma, \Omega_i}\right)^{\frac{1}{2}}.
\end{eqnarray*}
Consequently, there holds for all $k\geq 1$
\begin{equation}\label{EQ:April7:1-2-2}
\3bar \epsilon_h \3bar_w \leq C h^{k}\left((1+\tau^{-1})\sum_{i=1}^N \|u\|^2_{k, \Omega_i}+ \delta _{k, 1}h^{2\gamma}\|u\|^2_{1+\gamma, \Omega_i}\right)^{\frac{1}{2}}.
\end{equation}
As to the estimate for the error function $e_h$,
 we use the error equation (\ref{sehv}),  \eqref{EQ:April7:1-2-2}, the Cauchy-Schwarz inequality, and the triangle inequality to obtain
\begin{equation*}\label{EQ:April7:006}
\begin{split}
|b(e_h, \sigma)| &=|\ell_u(\sigma) - s(\epsilon_h, \sigma)| \leq |\ell_u(\sigma)| + \3bar \epsilon_h\3bar_w  \3bar \sigma\3bar_w\\
& \leq  Ch^{k}\left((1+\tau^{-1})\sum_{i=1}^N \|u\|^2_{k, \Omega_i}+ \delta _{k, 1}h^{2\gamma}\|u\|^2_{1+\gamma, \Omega_i}\right)^{\frac{1}{2}}\3bar \sigma \3bar_w,
\end{split}
\end{equation*}
which, combined with the \emph{inf-sup} condition (\ref{EQ:inf-sup-condition-02}), yields the following error estimate
\begin{equation}\label{EQ:April7:008}
\beta h^{1-\gamma} \|e_h\|_{1-\gamma} \leq C h^{k}\left((1+\tau^{-1})\sum_{i=1}^N \|u\|^2_{k, \Omega_i}+ \delta _{k, 1}h^{2\gamma}\|u\|^2_{1+\gamma, \Omega_i}\right)^{\frac{1}{2}}.
\end{equation}
Then the desired error estimate (\ref{erres}) follows from (\ref{EQ:April7:1-2-2}) and
(\ref{EQ:April7:008}). Finally, the estimate \eqref{erres:L2} is a direct result of \eqref{erres} and the triangle inequality. This completes the proof of the theorem.
\end{proof}

\begin{corollary} Assume $k\geq 1$.  Let $u$ and $(u_h;\lambda_h) \in M_h\times W_h^0$ be the exact solution of the elliptic interface problem \eqref{model}-\eqref{model2} and its numerical solution arising from the PDWG scheme (\ref{32})-(\ref{2}). Assume that the exact solution $u$ has the ``low" regularity of $u\in \prod_{i=1}^N  H^{\delta}(\Omega_i)$ for some $\frac{1}{2}< \delta \leq k$. Then, for $C^0$-WG elements, the following error estimate holds true:
 \begin{equation}\label{erres1-2}
\3bar \epsilon_h \3bar_w+h^{1-\gamma}\|e_h\|_{1-\gamma}\leq Ch^{\delta} (1+\tau^{-1})^{\frac{1}{2}}
\left(\sum_{i=1}^N \|u\|^2_{\delta, \Omega_i}\right)^{\frac{1}{2}},
\end{equation}
where $\frac{1}{2}< \gamma \leq 1$ is related to the regularity estimate \eqref{regu2}.
Consequently, one has the following error estimate
 \begin{equation}\label{erres:L2-2}
 \|u - u_h\|_{1-\gamma} \leq Ch^{\delta+\gamma-1} (1+\tau^{-1}) ^{\frac{1}{2}}\left(\sum_{i=1}^N \|u\|^2_{\delta, \Omega_i}\right)^{\frac{1}{2}}.
 \end{equation}
\end{corollary}

\begin{proof} It is easy to see $\epsilon_h\in W_h^0$. By letting $\sigma = \epsilon_h$  and $w=e_h$ in (\ref{sehv}) and (\ref{sehv2}) we arrive at
\begin{equation}\label{EQ:April7:001-1}
s(\epsilon_h, \epsilon_h) = \ell_u(\epsilon_h).
\end{equation}
Now, we use (\ref{EQ:April7:001-1}), the Cauchy-Schwarz inequality, the equation (\ref{lu2}), the trace inequality (\ref{trace-inequality}), and the estimates (\ref{3.2}) with $m=\delta$  to obtain
\begin{eqnarray*}
&&\3bar \epsilon_h \3bar_w^2 =s(\epsilon_h, \epsilon_h) = |\ell_u(\epsilon_h)| \\
&\leq & \Big(\sum_{T\in {\cal T}_h} h_T^{-1}\| a\nabla \epsilon_0\cdot \bn-\epsilon_n\|^2_{\partial T}\Big)^{\frac{1}{2}}
\Big(\sum_{T\in {\cal T}_h} h_T \|e_u
\|^2_{\partial T}\Big)^{\frac{1}{2}}\\
&&\ + \Big(\sum_{T\in {\cal T}_h}\tau \|\epsilon_0\|^2 \Big)^{\frac{1}{2}}
\Big(\sum_{T\in {\cal T}_h} \tau^{-1}\|ce_u\|_T^2 \Big)^{\frac{1}{2}}\\
&\leq& C\3bar \epsilon_h \3bar_w   \Big(\sum_{T\in {\cal T}_h} \|e_u
\|^2_{T}+h_T^{2\gamma} \|e_u\|^{2}_{\gamma, T}+ \tau^{-1}\|ce_u\|_T^2\Big)^{\frac{1}{2}}\\
&\leq&C \3bar \epsilon_h \3bar_w   h^{\delta} (1+\tau^{-1}) ^{\frac{1}{2}}
\left(\sum_{i=1}^N \|u\|^2_{\delta, \Omega_i}\right)^{\frac{1}{2}}.
\end{eqnarray*}
Hence,
\begin{equation}\label{EQ:April7:1-2-2-2}
\3bar \epsilon_h \3bar_w\leq Ch^{\delta} (1+\tau^{-1}) ^{\frac{1}{2}}\left(\sum_{i=1}^N \|u\|^2_{\delta, \Omega_i}\right)^{\frac{1}{2}}.
\end{equation}

As to the estimate for $e_h$, we use the error equation (\ref{sehv}), \eqref{EQ:April7:1-2-2-2}, the Cauchy-Schwarz inequality, and the triangle inequality to obtain
\begin{equation*}
\begin{split}
|b(e_h, \sigma)| &=|\ell_u(\sigma) - s(\epsilon_h, \sigma)| \leq |\ell_u(\sigma)| + \3bar \epsilon_h\3bar_w  \3bar \sigma\3bar_w\\
& \leq  Ch^{\delta} (1+\tau^{-1}) ^{\frac{1}{2}}\3bar \sigma \3bar_w \left(\sum_{i=1}^N \|u\|^2_{\delta, \Omega_i}\right)^{\frac{1}{2}},
\end{split}
\end{equation*}
which, combined with the \emph{inf-sup} condition (\ref{EQ:inf-sup-condition-02}), gives rise to the following error estimate
\begin{equation}\label{EQ:April7:008-2}
\beta h^{1-\gamma} \|e_h\|_{1-\gamma}   \leq  Ch^{\delta} (1+\tau^{-1}) ^{\frac{1}{2}}\left(\sum_{i=1}^N \|u\|^2_{\delta, \Omega_i}\right)^{\frac{1}{2}}.
\end{equation}
The desired error estimate (\ref{erres1-2}) then follows from (\ref{EQ:April7:1-2-2-2}) and
(\ref{EQ:April7:008-2}). Finally, \eqref{erres:L2-2} is a direct consequence of \eqref{erres1-2} and the triangle inequality. This completes the proof of the theorem.
\end{proof}

\section{Numerical Experiments}\label{Section:NE}
In this section, we will present some numerical results to verify the efficiency and accuracy of the proposed primal-dual weak Galerkin method \eqref{32}-\eqref{2} for solving the elliptic interface problem \eqref{model}-\eqref{model2}. In our experiments, we shall implement the algorithm with $k=1, 2$ in the finite element spaces $M_h$ and $W_h^0$. We shall compute various approximation errors for $u_h$ and $\lambda_h$, including the $L^2$ error $\|u-u_h\|_0$ and $\|\lambda_0\|_0$, the $H^1$ error $\|\lambda_0\|_1$ for $\lambda_h$, and the discrete error $\|\lambda_h\|_w$ as defined by \eqref{norm-new}. If not otherwise stated, the parameter $\tau$ in the PDWG numerical scheme will be $\tau=1$. The finite element partition $\T_h$ is obtained through a successive refinement of a coarse triangulation of the domain in aligning with the interface, by dividing each coarse element into four congruent sub-elements by connecting the mid-points of the three edges of the triangle. The right-hand side functions, the boundary and interface conditions are all derived from the exact solution.

{\it Example 1:} We consider the interface problem \eqref{model}-\eqref{model2} on the domain $\Omega=(0,1)^2$ with an interface given by $\Omega_1=[0.25,0.75]^2$ and
  $\Omega_2=\Omega\backslash\Omega_1$. The coefficients in the model equations are taken as
\[
     a_{1}=a_{2}=1, \ \ {\bf b}_1={\bf b}_2=(1,1),\ \ c_1=c_2=1.
\]
 The analytical solution to the elliptic equation is given as
 \[
   u=\left\{\begin{array}{ll}
  10-x^2-y^2 & \text{if} \ (x,y)\in\Omega_1,\\
  \sin({\pi x})\sin({\pi y}) & \text{if} \ (x,y)\in\Omega_2.
  \end{array}
  \right.
 \]
The initial mesh is shown in Figure \ref{fig:2} (left one). The mesh refinement of the previous level is done by connecting the mid-points of the edges. The mesh at the next level
is illustrated in Figure \ref{fig:2} (right one). The surface plot of the  PDWG solution $u_h$ on the finest mesh (i.e., after the fifth refinement of the initial mesh) is depicted in Figure \ref{fig:1}.

\begin{figure}[htb]
  \centering
  \includegraphics[width=.32\textwidth]{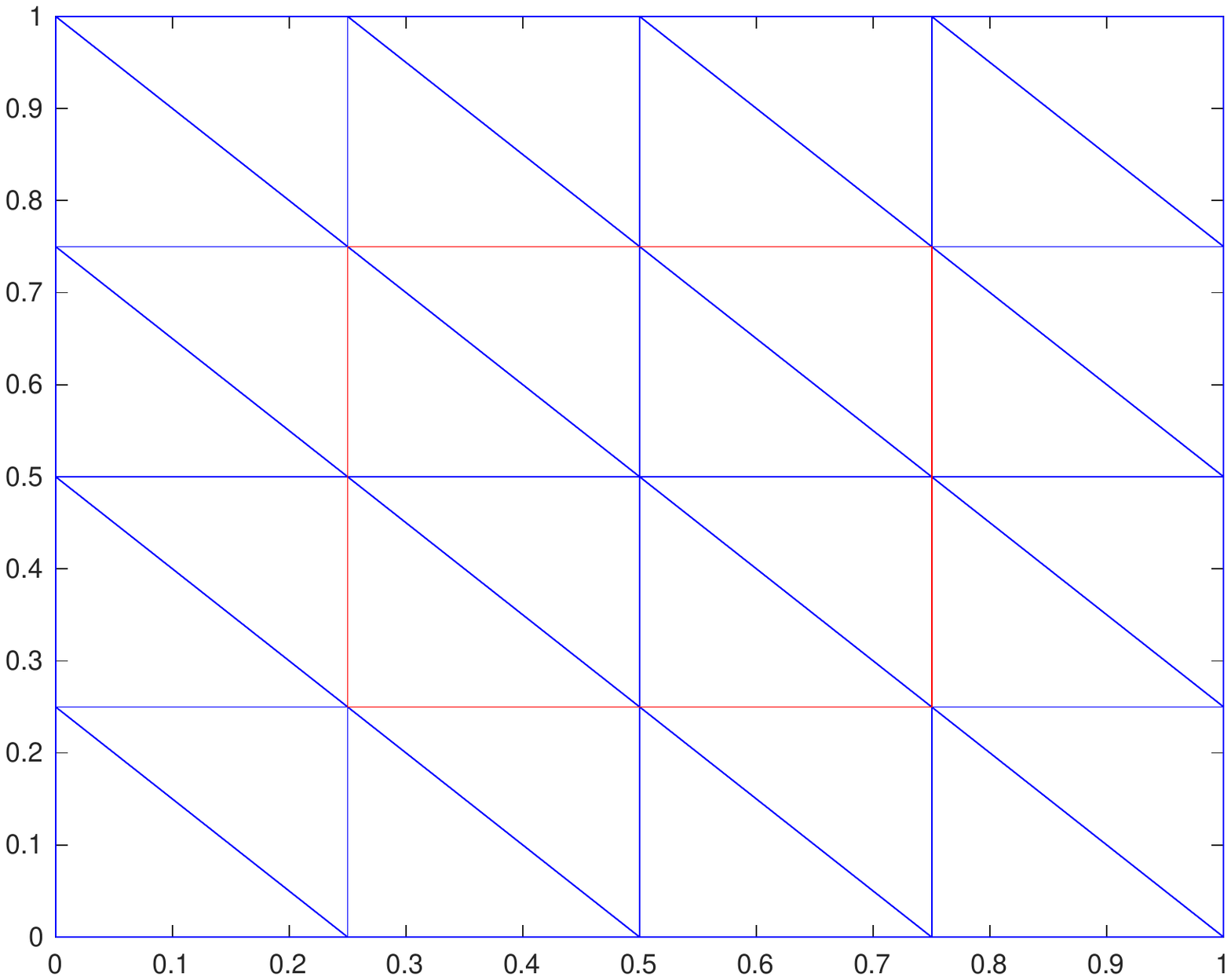}
  \includegraphics[width=.32\textwidth]{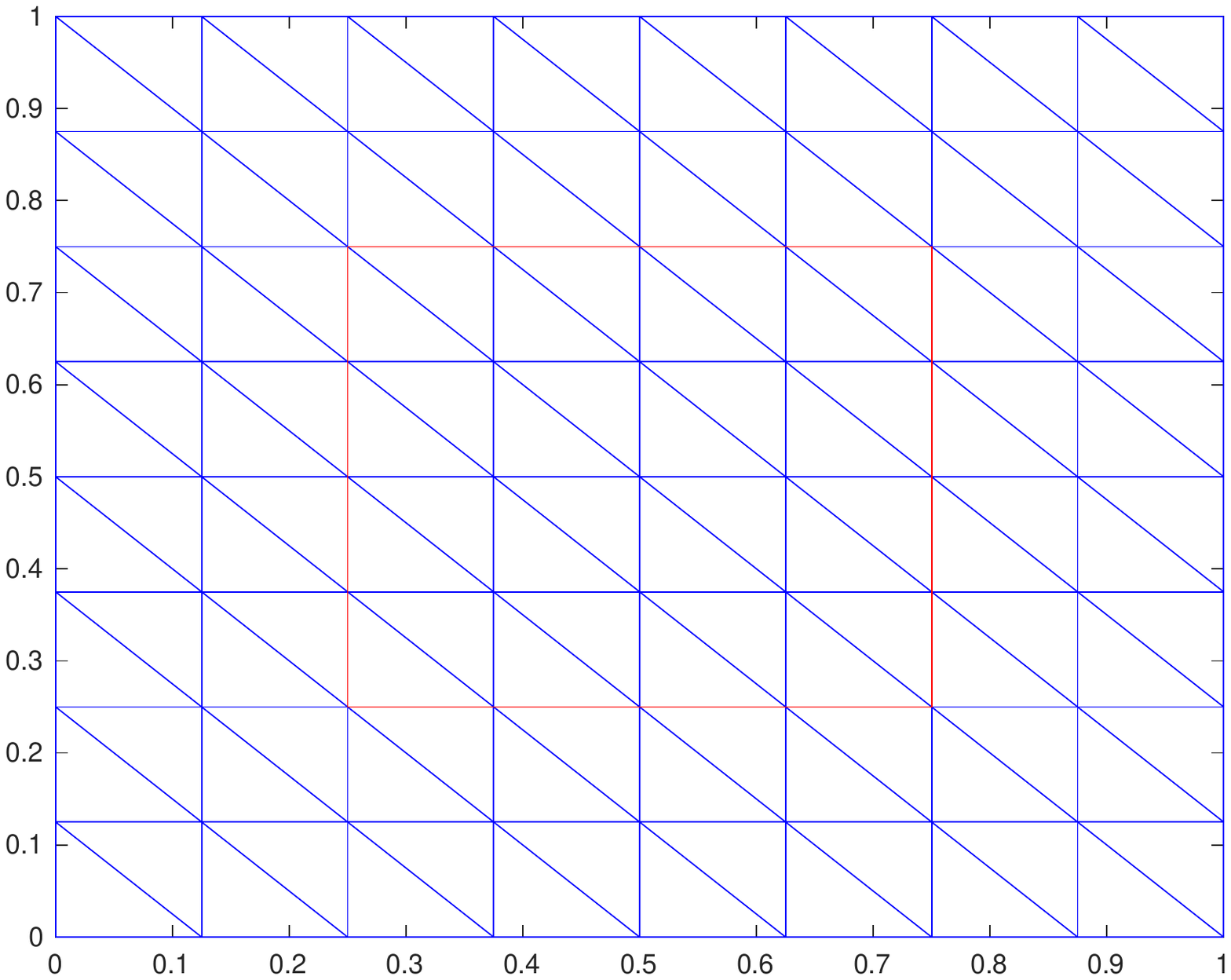}~
 \caption{The initial mesh (left) and the next level refinement (right).}
  \label{fig:2}
\end{figure}
\begin{figure}[htb]
  \centering
  \includegraphics[width=.6\textwidth]{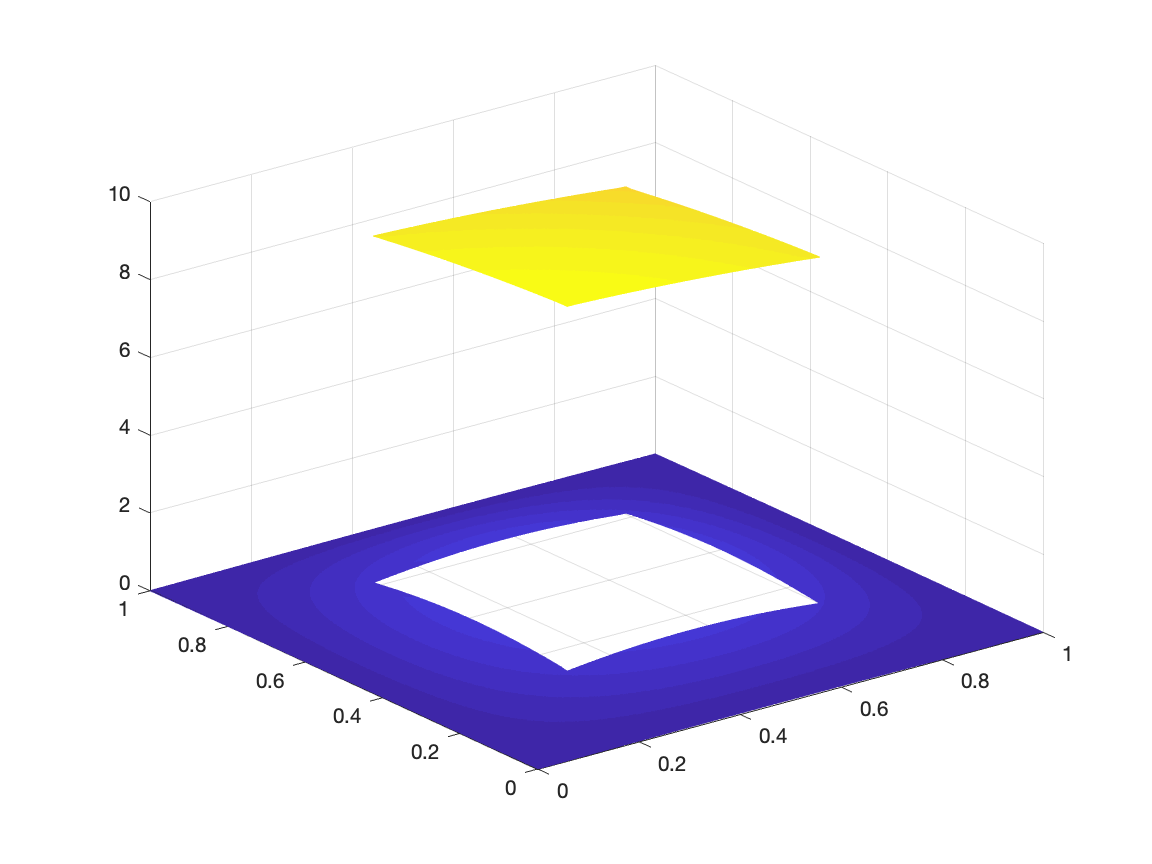}~
 \caption{Surface plot of the numerical solution $u_h$ calculated by the PDWG method of Example 1.}
  \label{fig:1}
\end{figure}

Table \ref{1} shows the numerical results and the rate of convergence
for $k=1,2$. We observe that, for both linear ($k=1$) and quadratic ($k=2$) PDWG methods, the
convergence rate for the errors $\|u-u_h\|_0$ and $\|\lambda_h\|_w$ is of $\mathcal{O}(h^{k})$,
which is consistent with the theoretical estimate \eqref{erres} in Theorem \ref{theoestimate}. As to the approximation error for $\lambda_0$, we observe a convergence rate of
$\mathcal{O}(h^{k+1})$ for $\|\lambda_0\|_1$ from this numerical experiment, which suggests a superconvergence for the dual variable $\lambda_0$ in the $H^1$-norm. We further observe a convergence of order $\mathcal{O}(h^{k+1})$ for $k=1$ and of $h^{k+2}$ for $k=2$ for $\lambda_0$
in the $L^2$ norm. Again, the $(k+2)$-th order of convergence for $\|\lambda_0\|_0$ indicates a pleasant superconvergence phenomenon of the PDWG method.

\begin{table}[htbp]\caption{Errors and  convergence rates of the linear and quadratic PDWG methods for Example 1.}\label{1}\centering
\begin{threeparttable}
        \begin{tabular}{c |c |c | c | c | c | c | c | c | c| }
        \hline
  & $h$ &   $\|\lambda_h\|_w$ & rate& $\|\lambda_0\|_1$  & rate & $\|\lambda_0\|_0$ & rate & $\|u-u_h\|_0$ & rate \\
       \hline \cline{2-10}
        &2.50e-1 &  1.15e-0 &   --   & 2.56e-1  &     -- &  1.21e-1 &  -- &  1.74e-1  &   --\\
        &1.25e-1 &  6.32e-1 &  0.87  & 9.52e-2  & 1.43 &  3.56e-2 &  1.77 &  7.78e-2  & 1.16\\
 $k=1$  &6.25e-2 &  3.34e-1 &  0.92  & 3.01e-2 &  1.66 &  9.73e-3 &  1.87 &  3.51e-2  & 1.15\\
        &3.13e-2 &  1.72e-1 &  0.96 &  8.49e-3  & 1.82 &  2.54e-3 &  1.94 &  1.66e-2  & 1.09\\
        &1.56e-2 &  8.73e-2 &  0.98 &  2.26e-3  & 1.92 &  6.47e-4 &  1.97  & 8.06e-3  & 1.04\\
  \hline
  & $h$ &   $\|\lambda_h\|_w$ & rate& $\|\lambda_0\|_1$  & rate & $\|\lambda_0\|_0$ & rate & $\|u-u_h\|_0$ & rate \\
       \hline \cline{2-10}
   &2.50e-1 &  2.00e-1 &     --      & 1.88e-2 &   --  & 5.71e-3&      --  & 2.80e-2  &      --\\
   &2.50e-1  & 5.12e-2 &  1.97 & 1.59e-3 &  3.56 &  3.61e-4 &  3.98 &  6.66e-3  & 2.07\\
 $k=2$  &6.25e-2  & 1.28e-2 &  2.00  & 1.54e-4 &  3.37 &  2.25e-5 &  4.00  & 1.64e-3  & 2.02\\
   &3.13e-2&  3.19e-3 &  2.00 &  1.67e-5 &  3.20 &  1.40e-6 &  4.01  & 4.08e-4 &  2.01\\
   &1.56e-2 &  7.97e-4 &  2.00  & 1.94e-6 &  3.10 &  8.72e-8 &  4.00  & 1.02e-4 &  2.00\\
   \hline
 \end{tabular}
 \end{threeparttable}
\end{table}

{\it Example 2:} We consider an circular interface problem on the domain $\Omega=(0,1)^2$.  Here $\Omega_1$ is the disc centered at the point $(0.5,0.5)$
with radius $r=0.25$, and $\Omega_2=\Omega\backslash\Omega_1$.
The coefficients are taken as
\[
     a_{1}=a_{2}=2+\sin(x+y), \ \ {\bf b}_1={\bf b}_2=(x,y),\ \ c_1=c_2=4+x.
\]
 The analytical solution to the interface problem is
 \[
   u=\left\{\begin{array}{ll}
  \sin(x+y)+\cos(x+y)+5, & \text{if} \ (x,y)\in\Omega_1,\\
  x+y+1, & \text{if} \ (x,y)\in\Omega_2.
  \end{array}
  \right.
 \]
We plot in Figure \ref{fig:3} the interface and subdomains (left), the
initial mesh (middle), and the refined mesh generated from twice refinement of the initial mesh (right), respectively. The surface plot of the approximate solution $u_h$ calculated by the PDWG method with $k=1$ on the finest mesh is shown in Figure \ref{fig:4}.
\begin{figure}[htb]
  \centering
   \includegraphics[width=.32\textwidth]{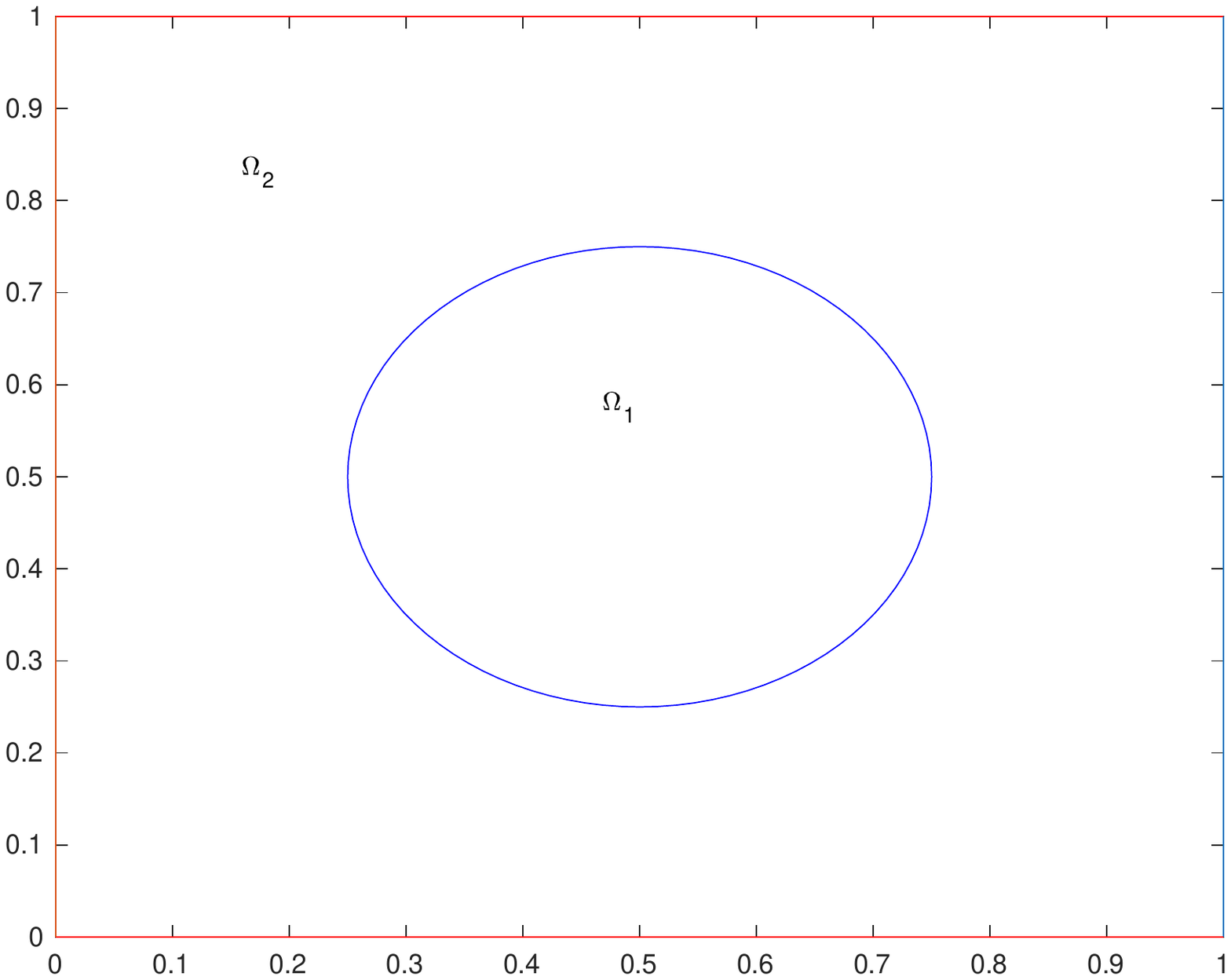}
  \includegraphics[width=.32\textwidth]{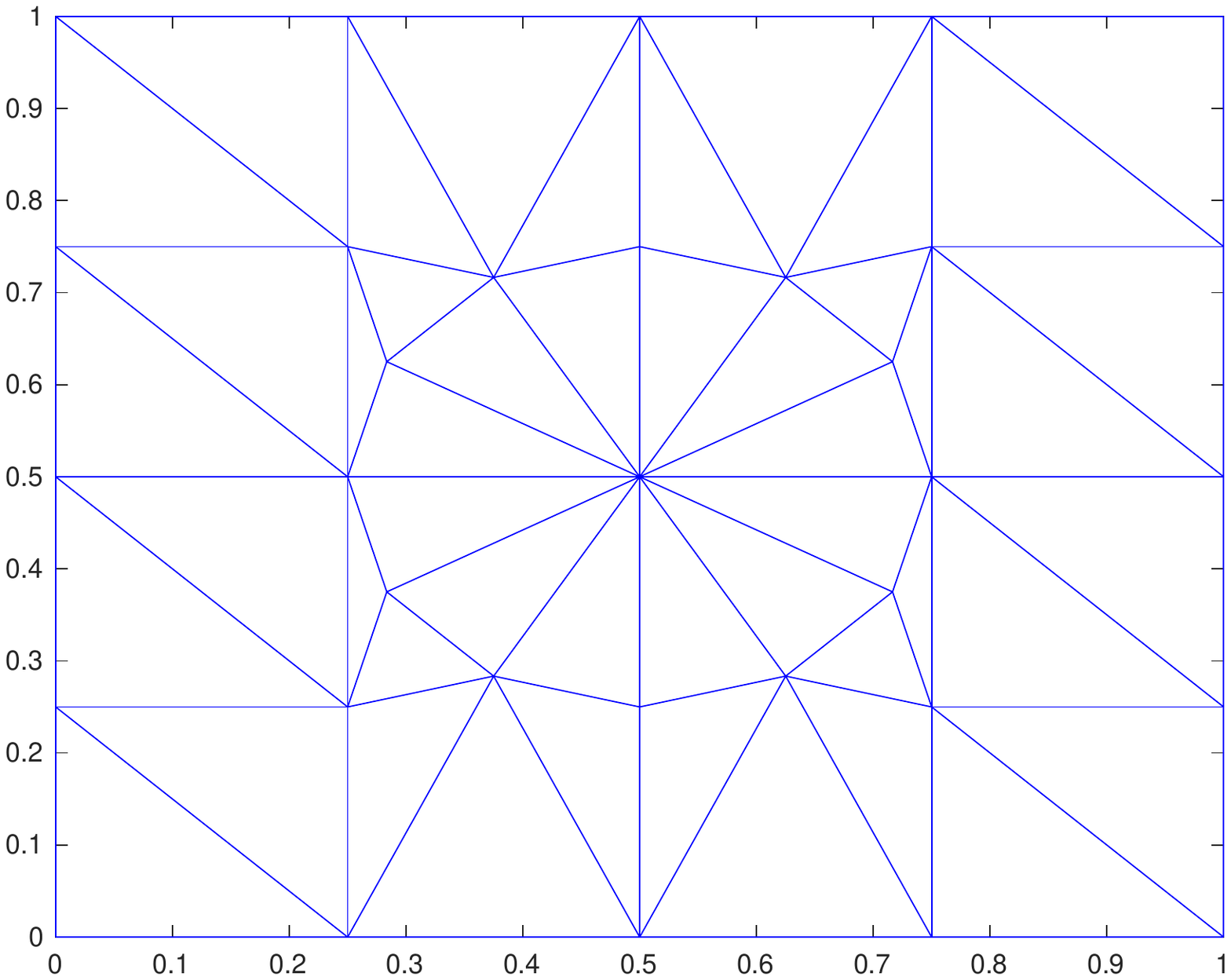}~
  \includegraphics[width=.32\textwidth]{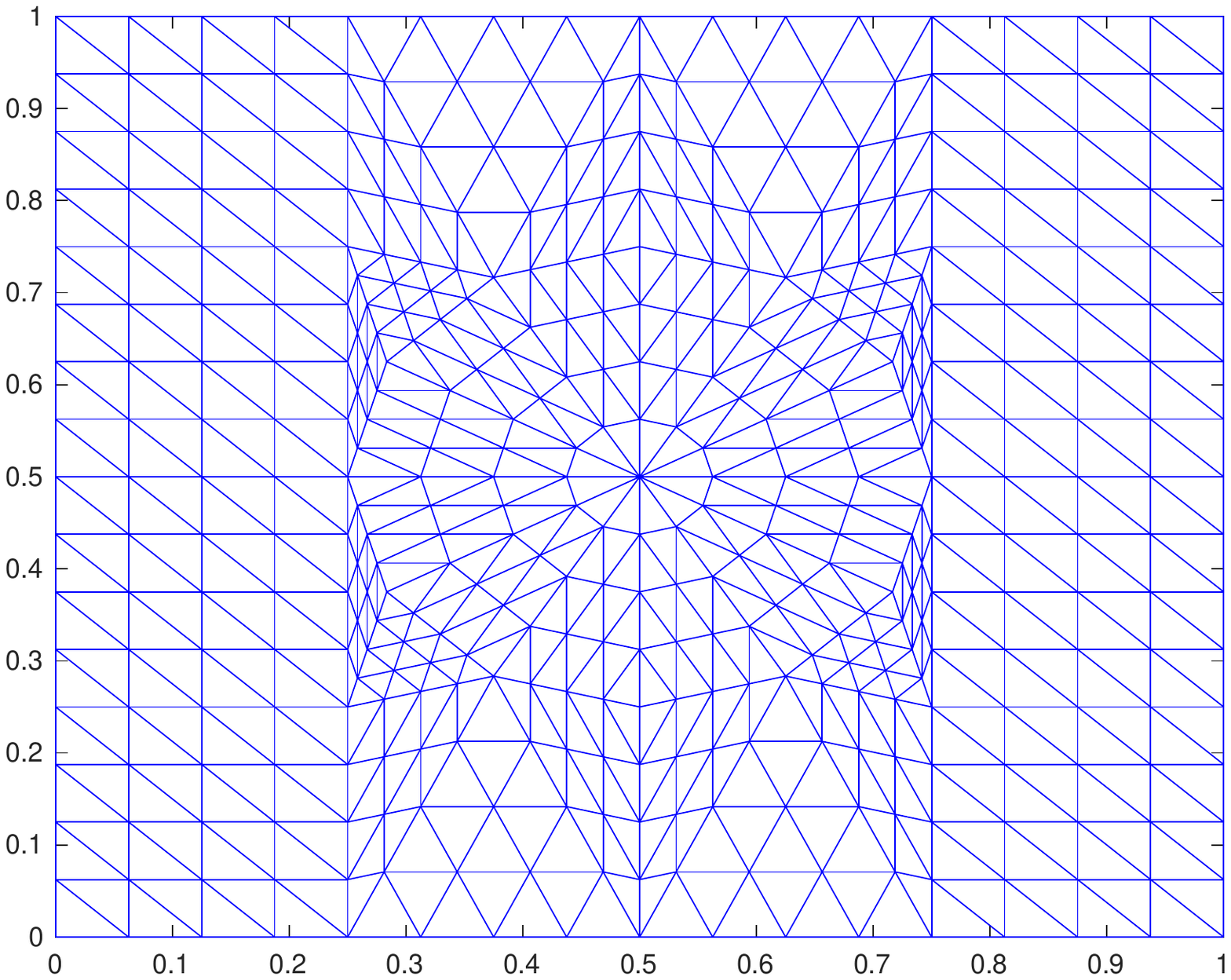}
 \caption{The  interface and subdomains (left) and the initial mesh (middle) and the refined mesh from the twice refinement of  the initial mesh (right).}
  \label{fig:3}
\end{figure}

\begin{figure}[htb]
  \centering
  \includegraphics[width=.6\textwidth]{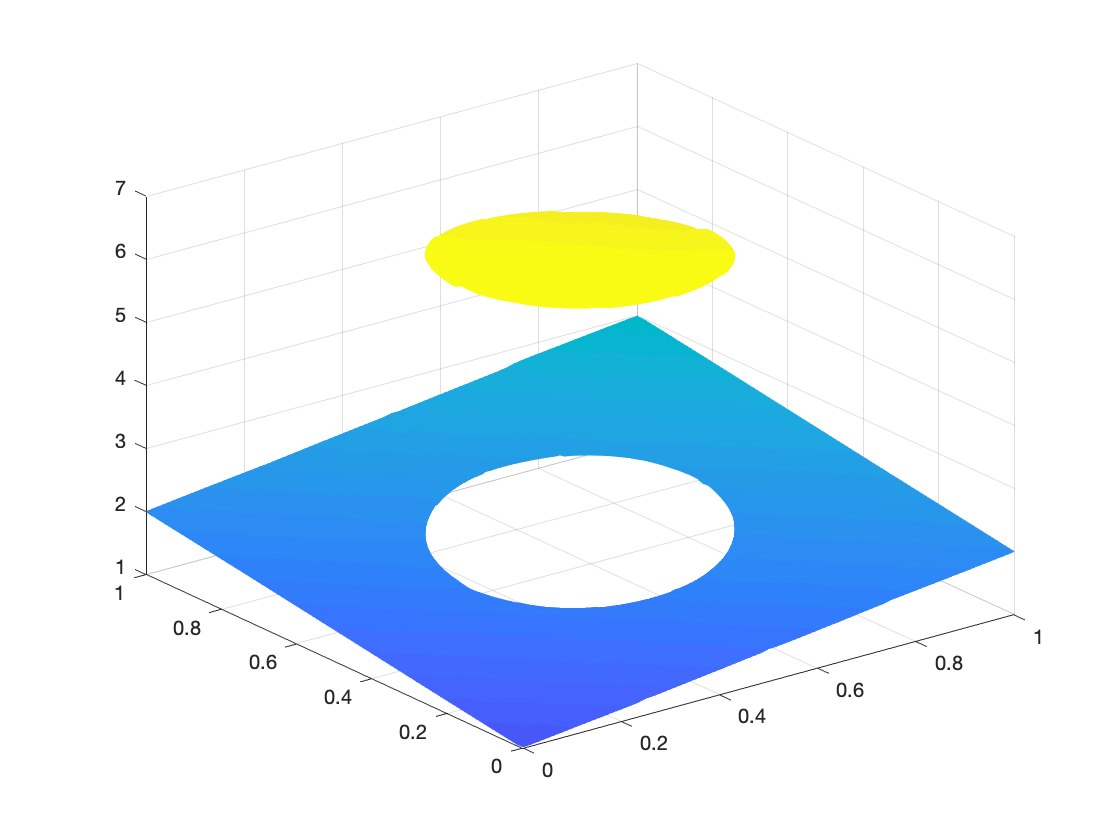}~
 \caption{Surface plot of the numerical solution $u_h$ calculated by the PDWG method of Example 2.}
  \label{fig:4}
\end{figure}

We present in Table \ref{3} the approximation errors and corresponding convergence rates for the primal variable $u_h$ and dual variable $\lambda_h$, from which
we observe a convergence rate of $\mathcal{O}(h^k)$ for both $\|u-u_h\|_0$ and $\|\lambda_h\|_w$. In other words, the error bound given in \eqref{erres} is sharp. Analogous to Example 1, we see the error $\|\lambda_0\|_1$ converges to zero with an order of $k+1$ for both linear and quadratic PDWG methods. Table \ref{3} also shows a convergence of  $\lambda_0$
with order $(k+1)$ for $k=1$ and $(k+2)$ for $k=2$ in $L^2$ norm.

\begin{table}[htbp]\caption{Errors and  convergence rates of the linear and quadratic PDWG methods for Example 2.}\label{3}\centering
\begin{threeparttable}
        \begin{tabular}{c |c |c | c | c | c | c | c | c | c| }
        \hline
  & $h$ &   $\|\lambda_h\|_w$ & rate& $\|\lambda_0\|_1$  & rate & $\|\lambda_0\|_0$ & rate & $\|u-u_h\|_0$ & rate \\
       \hline \cline{2-10}
   &1.2941e-1 &  6.43e-1  &     -- &  3.02e-1  &   --  & 2.15e-2 &    --  & 7.50e-2  &    -- \\
   &6.4705e-2  & 3.55e-1 &  0.86 &  9.31e-2  & 1.70 &  5.17e-3 &  2.06  & 3.75e-2 &  0.99\\
  $k=1$ &3.2352e-2  & 1.98e-1  & 0.84 &  2.97e-2  & 1.65 &  1.42e-3 &  1.87 &  1.77e-2 &  1.09\\
   &1.6176e-2  & 1.07e-1 &  0.89 &  8.85e-3  & 1.75 &  3.80e-4 &  1.90  & 8.35e-3 &  1.09\\
   &8.0881e-3  & 5.59e-2 &  0.94 &  2.43e-3  & 1.87  & 9.98e-5  & 1.93  & 4.04e-3 &  1.05\\
 \hline
  & $h$ &   $\|\lambda_h\|_w$ & rate& $\|\lambda_0\|_1$  & rate & $\|\lambda_0\|_0$ & rate & $\|u-u_h\|_0$ & rate \\
       \hline \cline{2-10}
   &1.2941e-1 &  4.51e-2 &      -- &  6.20e-4  &    -- &  5.44e-4  &     -- &  3.28e-3 &      -- \\
   &6.4705e-2 &  1.14e-2 &  1.99 &  5.48e-5  & 3.50 &  3.43e-5  & 4.00 &  8.30e-4 &  1.98\\
 $k=2$  &3.2352e-2 &  2.94e-3 &  1.95 &  5.20e-6  & 3.40 &  2.15e-6  & 4.00 &  2.08e-4 &  2.00\\
   &1.6176e-2 &  7.24e-4 &  2.02 &  5.38e-7  & 3.27 &  1.35e-7  & 4.00  & 5.17e-5  & 2.01\\
  & 8.0881e-3 &  1.84e-4 &  1.97 &  6.20e-8  & 3.12 &  8.49e-9  & 3.99  & 1.30e-5  & 2.00\\
  \hline
 \end{tabular}
 \end{threeparttable}
\end{table}

{\it Example 3:} The interface problem \eqref{model}-\eqref{model2} is defined on the domain $\Omega=(0,1)^2$ with a closed interface $\Gamma$ parameterized as follows
\[
    r=0.5+\frac{3\sin(3\theta)}{4}.
\]
The subdomain $\Omega_1$ is given by the region bounded by the curve $\Gamma$ and $\Omega_2=\Omega\backslash\Omega_1$ is the portion of the domain outside $\Gamma$.
The PDE coefficients are given by
\[
     a_{1}=1+x+y, \ \ a_{2}=1, \ \ {\bf b}_1={\bf b}_2=(1,1+y),\ \ c_1=c_2=2.
\]
 The exact solution to the elliptic problem is given as
 \[
   u=\left\{\begin{array}{ll}
  e^x\cos(y)+10, & \text{if} \ (x,y)\in\Omega_1,\\
  5e^{-x^2-y^2}, & \text{if} \ (x,y)\in\Omega_2.
  \end{array}
  \right.
 \]

The interface and subdomains, the initial mesh, and the refined mesh after two successive refinements of the initial mesh are shown in Figure \ref{fig:5}. The numerical solution $u_h$
calculated by the PDWG method with $k=2$ on the refined mesh are depicted in Figure \ref{fig:6}.
The numerical errors of the linear and quadratic PDWG methods are reported in Table \ref{4}. It can be seen that the theoretical convergence (i.e., $\mathcal{O}(h^{k})$ for both $\|u-u_h\|_0$ and $\|\lambda_h\|_w$) is achieved in this numerical test. Moreover, a convergence of
$\mathcal{O}(h^{k+1})$ for $\|\lambda_0\|_1$, and $\mathcal{O}(h^{k+1})$ and $\mathcal{O}(h^{k+2})$   for the error $\|\lambda_0\|_0$ is observed for the case of $k=1$ and $k=2$, respectively.

\begin{figure}[htb]
  \centering
   \includegraphics[width=.32\textwidth]{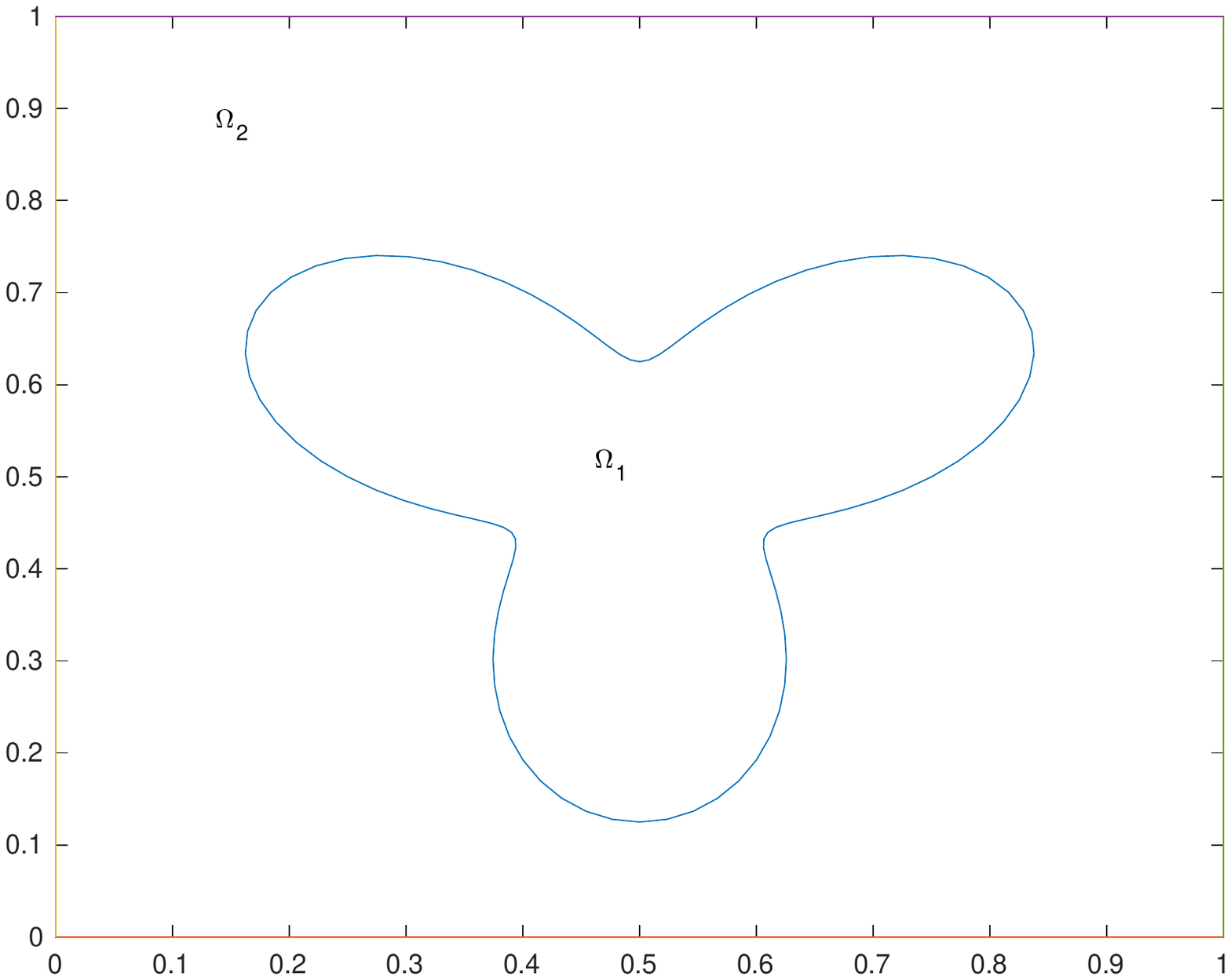}
  \includegraphics[width=.32\textwidth]{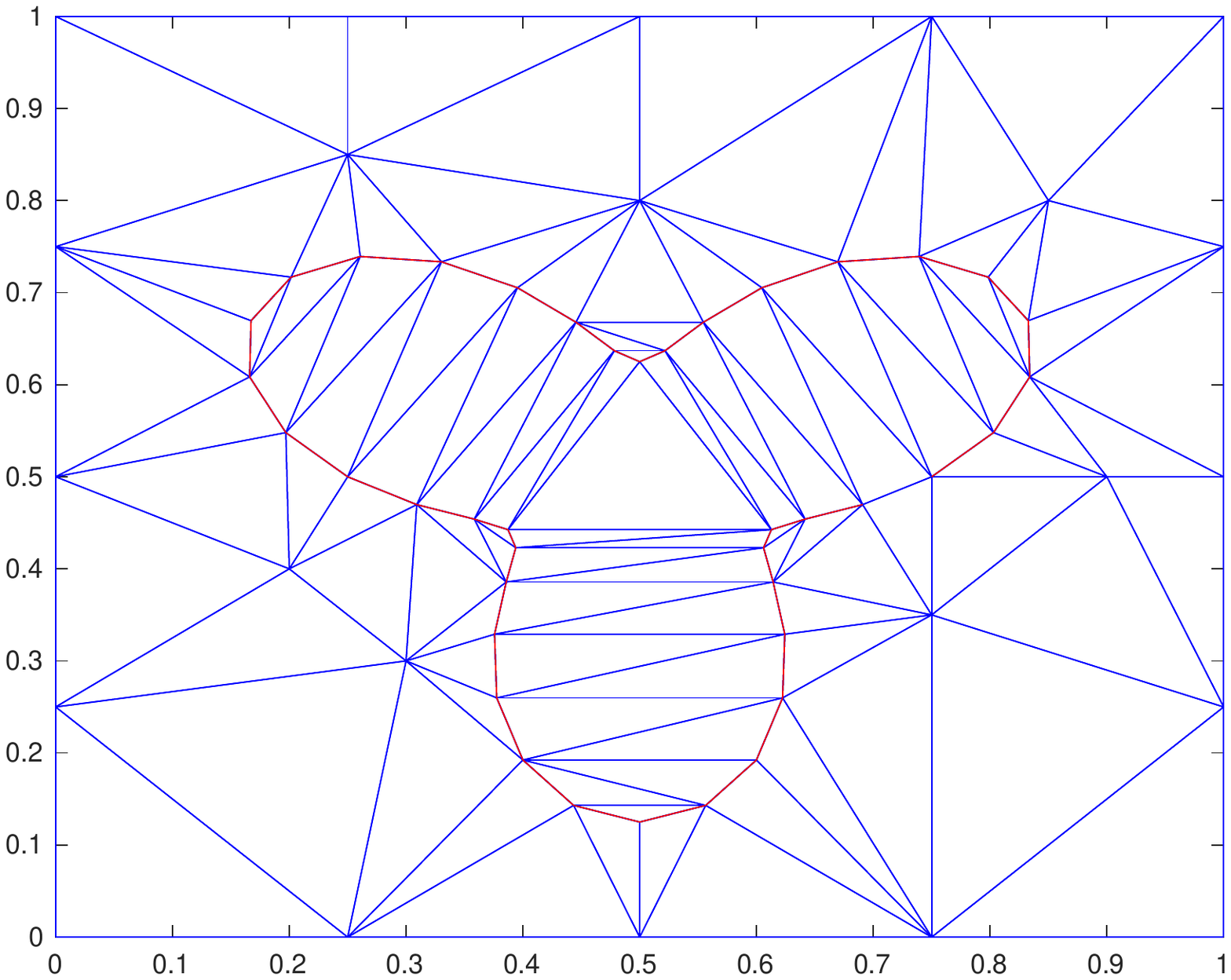}~
  \includegraphics[width=.32\textwidth]{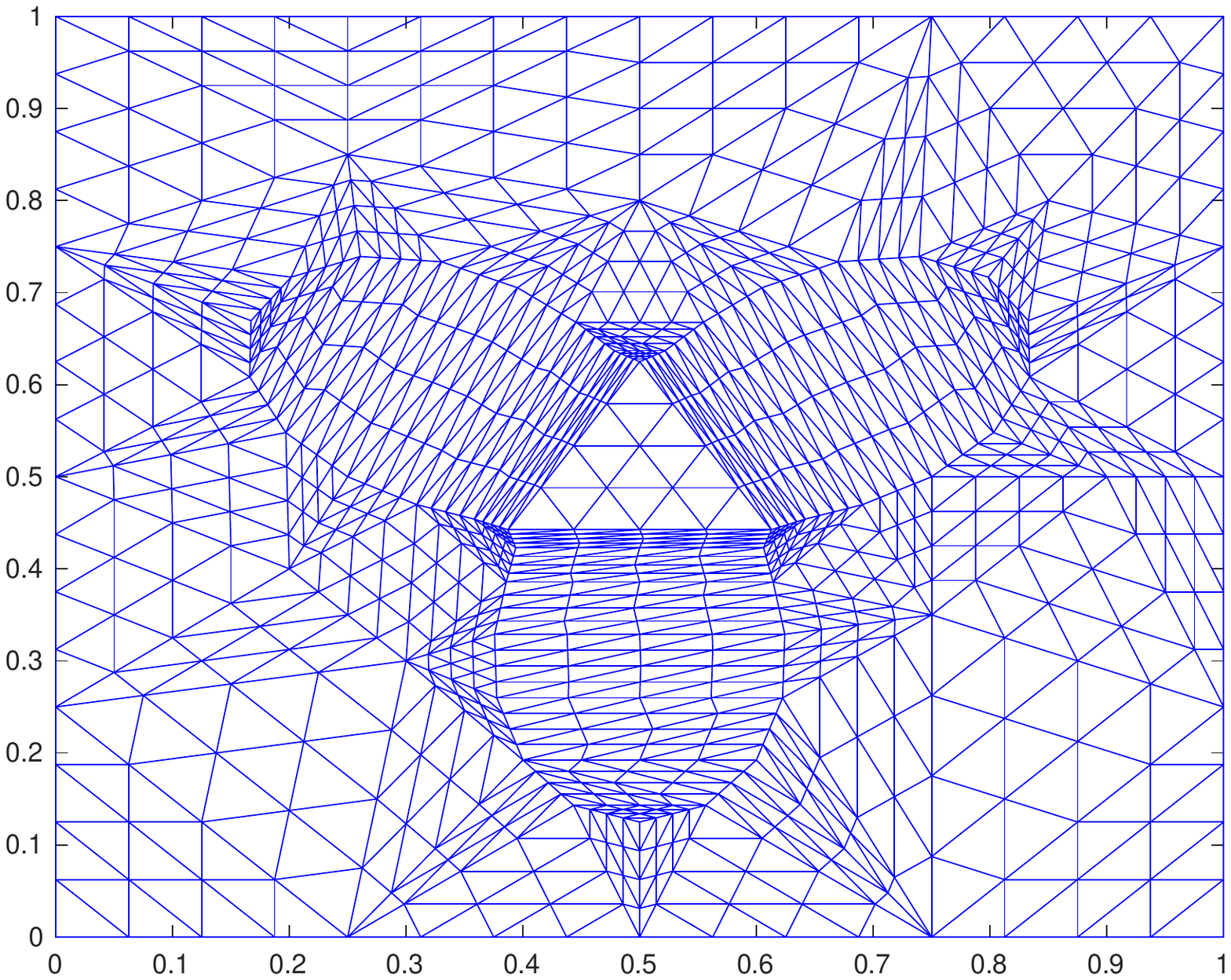}
\caption{The  interface and subdomains (left) and the initial mesh (middle) and the refined mesh (right).}
  \label{fig:5}
\end{figure}

\begin{figure}[htb]
  \centering
  \includegraphics[width=.6\textwidth]{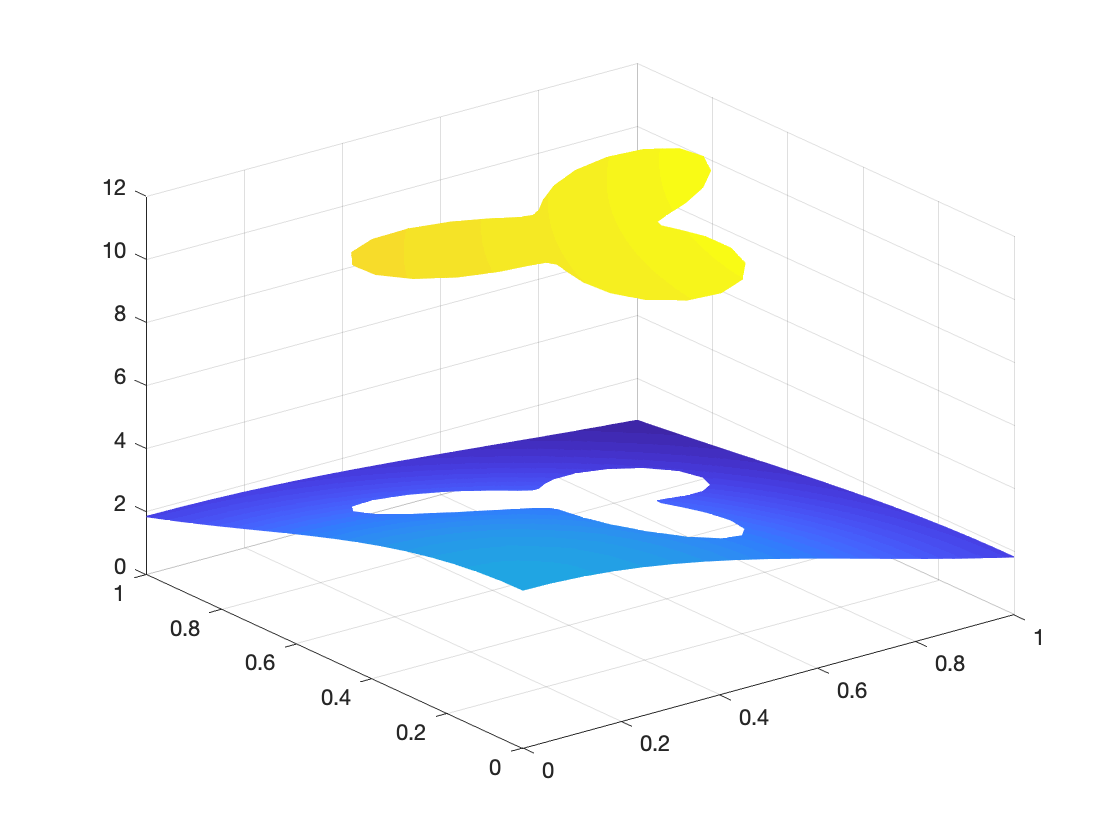}~
 \caption{Surface plot of the numerical solution $u_h$ calculated by the PDWG method of Example 3.}
  \label{fig:6}
\end{figure}

\begin{table}[htbp]\caption{Errors and  convergence rates of the linear and quadratic PDWG methods for Example 3.}\label{4}\centering
\begin{threeparttable}
        \begin{tabular}{c |c |c | c | c | c | c | c | c | c| }
        \hline
  & $h$ &   $\|\lambda_h\|_w$ & rate& $\|\lambda_0\|_1$  & rate & $\|\lambda_0\|_0$ & rate & $\|u-u_h\|_0$ & rate \\
       \hline \cline{2-10}

   &3.5355e-1 &  1.35e-0 &    -- &  3.77e-1  &  --  & 6.52e-2 &     --    &  2.38e-1  &     --  \\
   &1.7678e-1 &  7.32e-1 &  0.81 &  1.26e-1 &  1.58  & 1.65e-2 &  1.98 &  1.25e-1  & 0.93  \\
 $k=1$  &8.8388e-2 &  4.01e-1 &  0.87 &  4.14e-2 &  1.61  & 4.37e-3  & 1.92 &  6.63e-2  &  0.92 \\
   &4.4194e-2 &  2.16e-1 &  0.89 &  1.29e-2 &  1.68 &  1.17e-3  & 1.90 &  3.54e-2  & 0.91 \\
  & 2.2097e-2 &  1.14e-1 &  0.92 &  3.75e-3 &  1.78 &  3.07e-4  & 1.93 &  1.90e-2  & 0.90 \\
  \hline
 & $h$ &   $\|\lambda_h\|_w$ & rate& $\|\lambda_0\|_1$  & rate & $\|\lambda_0\|_0$ & rate & $\|u-u_h\|_0$ & rate \\
       \hline \cline{2-10}
       & 3.5355e-1 &  2.01e-1 &     -- &  1.09e-2  &    -- &  5.82e-3  &   -- &  1.66e-2 &  --\\
       & 1.7678e-1 &  5.12e-2 &  1.97 &  9.70e-4  & 3.50  & 3.76e-4 &  3.95  & 4.13e-3 &  2.01\\
  $k=2$ & 8.8388e-2 &  1.28e-2 &  2.00 &  9.04e-5  & 3.42 &  2.37e-5 &  3.99 &  1.02e-3 &  2.02\\
        &4.4194e-2 &  3.21e-3 &  2.00 &  9.70e-6  & 3.22 &  1.48e-6 &  4.00  & 2.53e-4 &  2.01\\
       & 2.2097e-2&   8.02e-4 &  2.00 &  1.14e-6  & 3.01 &  9.28e-8 &  4.00  & 6.30e-5 &  2.00\\
   \hline
 \end{tabular}
 \end{threeparttable}
\end{table}

{\it Example 4:}
The interface $\Gamma$ on the domain $\Omega=(0,1)^2$ is characterized by the following equation in the polar coordinates:
\begin{eqnarray*}
   && x(\theta)=(1/2+1/2\cos(m\theta)\sin(n\theta))\cos(\theta),\\
   &&  y(\theta)=(1/2+1/2\cos(m\theta)\sin(n\theta))\sin(\theta),
\end{eqnarray*}
where $m=2$ and $n=6$. The subdomain $\Omega_1$ is the region inside $\Gamma$ and $\Omega_2=\Omega\backslash\Omega_1$. The coefficients in the elliptic interface problem are given by
\[
     a_{1}=(xy+2)/5, \ \ a_{2}=(x^2-y^2+3)/7, \ \ {\bf b}_1=(0,1),\ \ {\bf b}_2=(1,0),\ \ c_1=2,\ c_2=1.
\]
 The exact solution to the interface problem is
 \[
   u=\left\{\begin{array}{ll}
  x+y+2, & \text{if} \ (x,y)\in\Omega_1,\\
  0.5\sin(x+y)+0.5\cos(x+y)+0.3, & \text{if} \ (x,y)\in\Omega_2.
  \end{array}
  \right.
 \]

The interface and subdomains, the initial mesh, and the refined mesh after two successive refinement of the initial mesh are shown in Figure \ref{fig:7}. The PDWG solution $u_h$
on the finest mesh are depicted in Figure \ref{fig:8}. The numerical errors of the linear and quadratic PDWG method are reported in Table \ref{5}. The numerical convergence rate for
$\|\lambda_h\|_w, \|u-u_h\|_0$, and $\|\lambda_0\|_1$ are seen to be $\mathcal{O}(h^{k}), \mathcal{O}(h^{k}), \mathcal{O}(h^{k+1})$, respectively. Once again, the numerical experiment suggests a convergence at the optimal order of $\mathcal{O}(h^{k+1})$ for $\|\lambda_0\|_0$ for the linear PDWG method and a superconvergence of $\mathcal{O}(h^{k+2})$ for the quadratic PDWG method.

 \begin{figure}[htb]
  \centering
   \includegraphics[width=.32\textwidth]{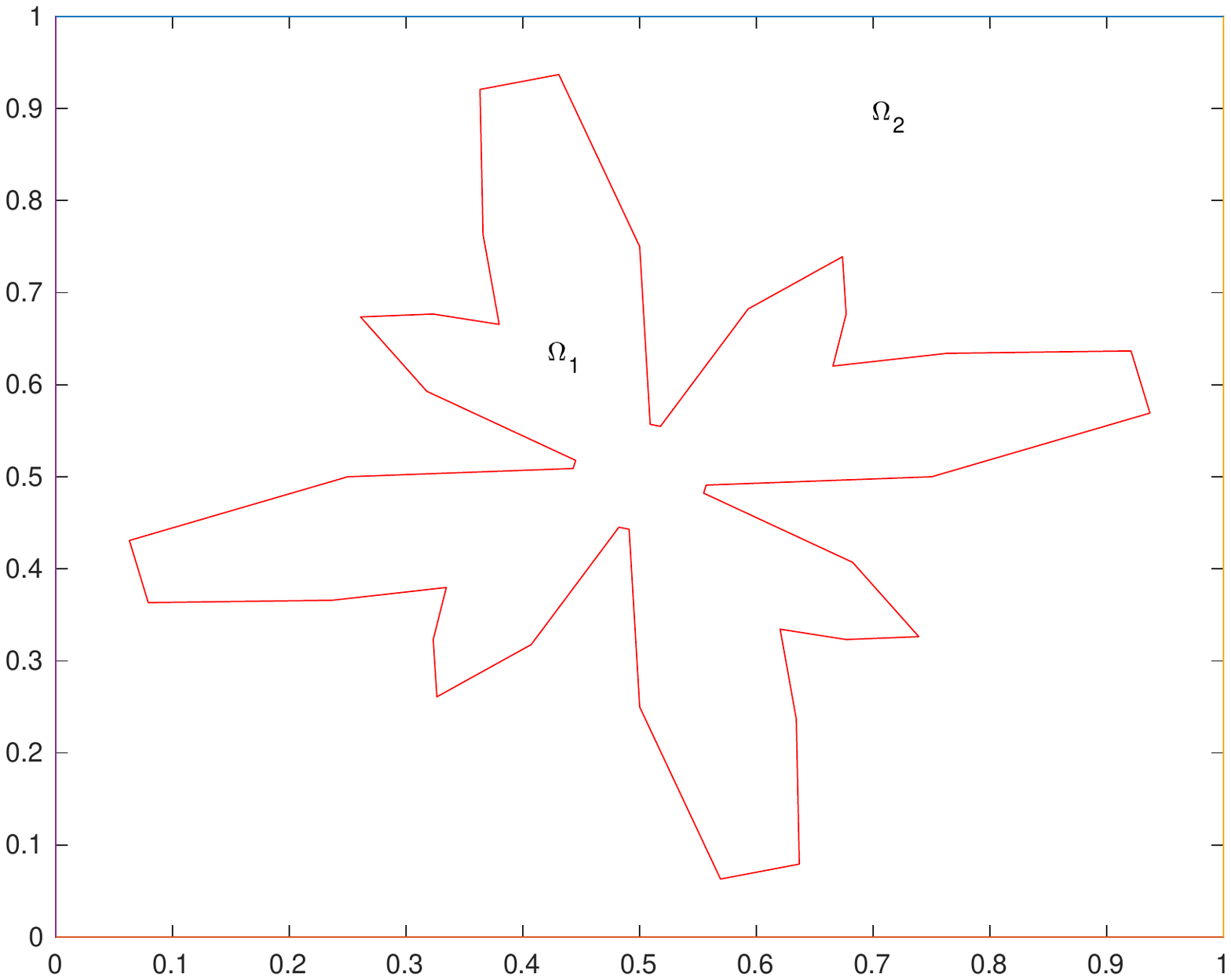}
  \includegraphics[width=.32\textwidth]{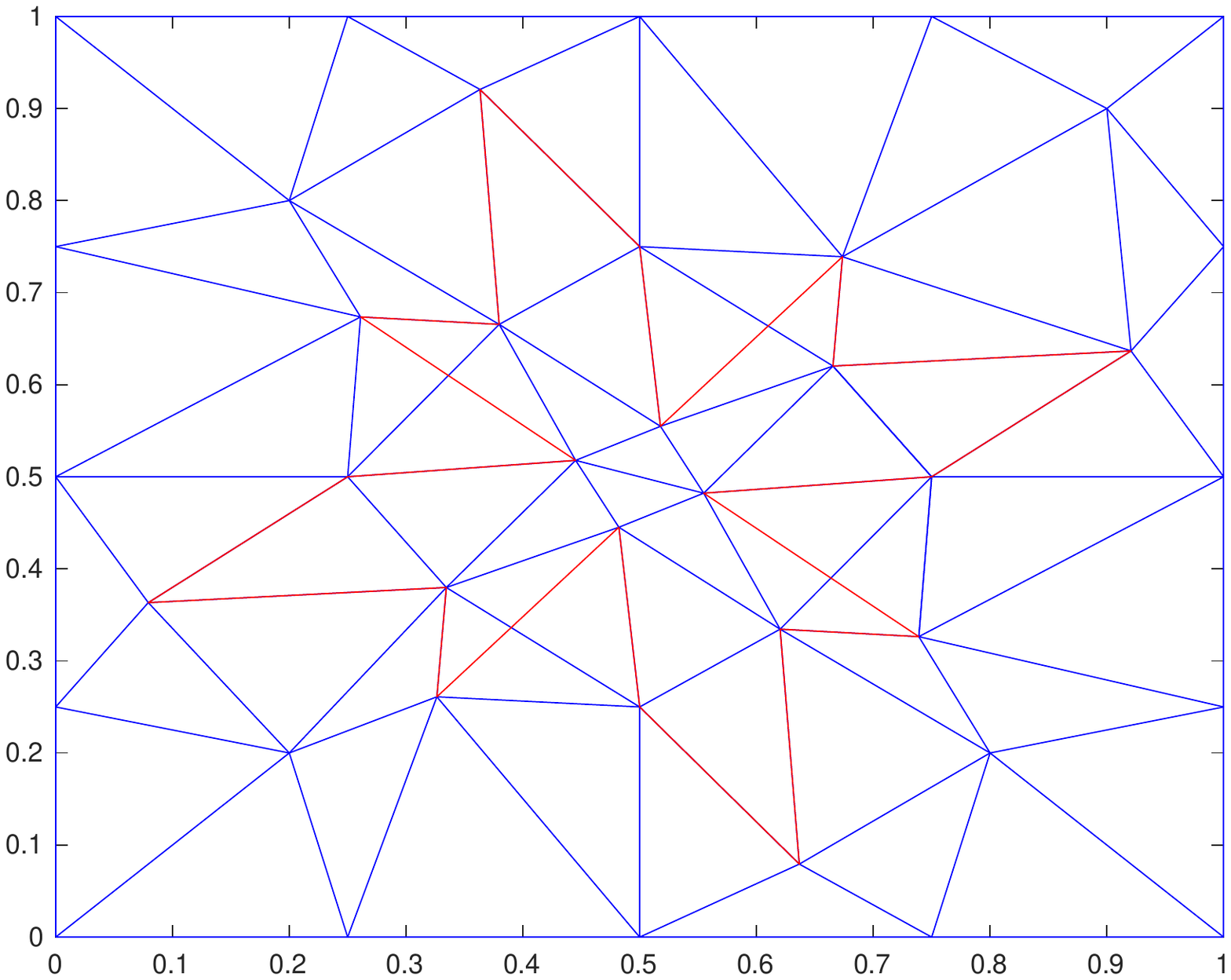}~
  \includegraphics[width=.32\textwidth]{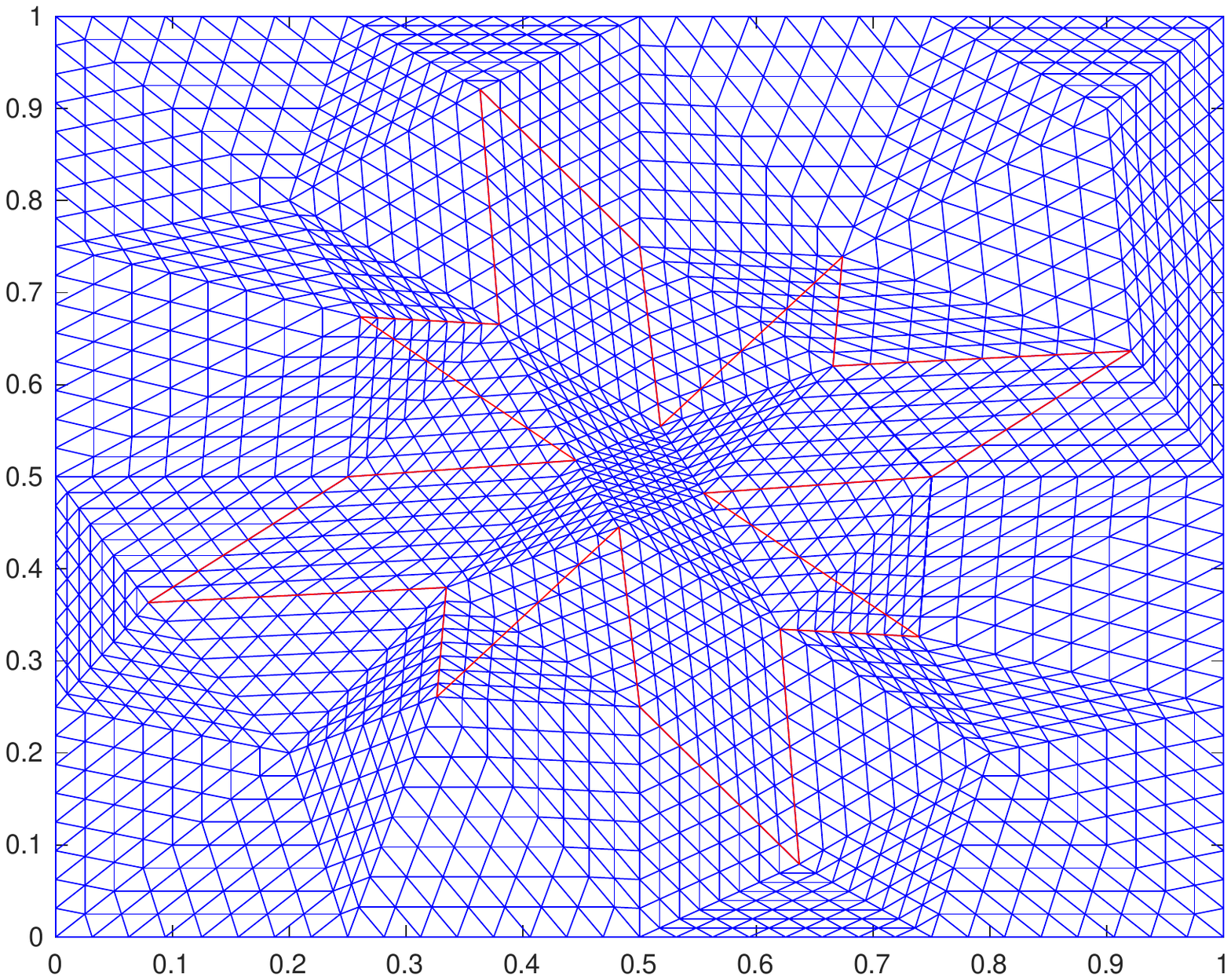}
\caption{ The  interface and subdomains (left) and the initial mesh (middle) and the refined mesh (right).}
  \label{fig:7}
\end{figure}

\begin{figure}[htb]
  \centering
  \includegraphics[width=.6\textwidth]{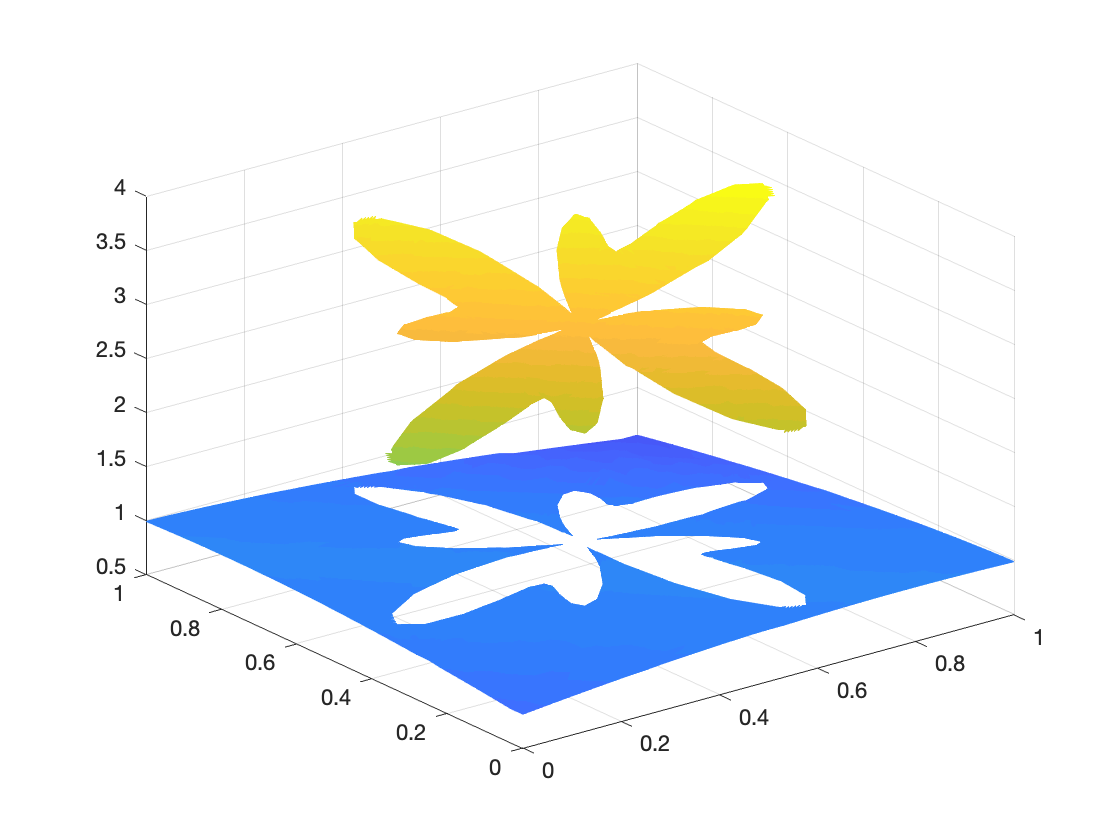}~
 \caption{ Surface plot of the numerical solution $u_h$ calculated by the PDWG method of Example 4.}
  \label{fig:8}
\end{figure}

\begin{table}[htbp]\caption{Errors and  convergence rates of the linear and quadratic PDWG methods for Example 3.}\label{5}\centering
\begin{threeparttable}
        \begin{tabular}{c |c |c | c | c | c | c | c | c | c| }
        \hline
  & $h$ &   $\|\lambda_h\|_w$ & rate& $\|\lambda_0\|_1$  & rate & $\|\lambda_0\|_0$ & rate & $\|u-u_h\|_0$ & rate \\
       \hline \cline{2-10}
        &3.5355e-1 &  1.58e-1 &   --  &  2.94e-2   &       --   & 7.35e-3 &     --  &  4.02e-2  &    --\\
        &1.7678e-1 &  8.12e-2 &  0.96   & 9.08e-3  &  1.70   & 1.79e-3   & 2.04   & 1.94e-2  &  1.05\\
  $k=1$ & 8.8388e-2&   4.15e-2&   0.97  &  2.93e-3 &   1.63  &  4.59e-4  &  1.97 &   9.55e-3  &  1.03\\
        &4.4194e-2 &  2.14e-2 &  0.96 &  8.58e-4   & 1.77  &  1.20e-4 &   1.93   & 4.77e-3   & 1.00\\
        &2.2097e-2 &  1.09e-2 &  0.97  &  2.35e-4  &  1.87  &  3.06e-5  &  1.98  &  2.51e-3  &  0.93\\
  \hline
& $h$ &   $\|\lambda_h\|_w$ & rate& $\|\lambda_0\|_1$  & rate & $\|\lambda_0\|_0$ & rate & $\|u-u_h\|_0$ & rate \\
\hline \cline{2-10}
 &3.5355e-01 &  1.42e-2  &      --  & 1.78e-3    &       --  & 3.11e-4   &    --  & 2.25e-3  &  --\\
 &  1.7678e-01 &  3.59e-3 &  1.99&   1.93e-4 &  3.21 &  1.97e-5 &  3.98 &  5.53e-4 &  2.02\\
$k=2$ &  8.8388e-02 &  8.97e-4&   2.00 &  2.23e-5 &  3.11 &  1.24e-6 &  3.99 &  1.36e-4 &  2.02\\
 &  4.4194e-02 &  2.24e-4 &  2.00 &  2.70e-6 &  3.05 &  7.76e-8 &  4.00&   3.39e-5 &  2.01\\
 &  2.2097e-02 &  5.61e-5 &  2.00  & 3.31e-7 &  3.03 &  4.86e-9 &  4.00 &  8.44e-6 &  2.01\\
  \hline
 \end{tabular}
 \end{threeparttable}
\end{table}

{\it Example 5:} We consider an interface problem on the domain $\Omega=(0,1)^2$ with
an interface $\Gamma$ parameterized in the polar angle $\theta$ as follows
\[
    r=0.5+\frac{\sin(5\theta)}{7}.
\]
The subdomain $\Omega_1$ is the part inside $\Gamma$ and $\Omega_2=\Omega\backslash\Omega_1$ is the  part outside $\Gamma$. The coefficients in the PDE are given by
\[
     a_{1}=0.01, \ \ a_{2}=0.1, \ \ {\bf b}_1= {\bf b}_2=(0,0),\ \ c_1=c_2=0.
\]
The exact solution to the elliptic interface problem is
 \[
   u=\left\{\begin{array}{ll}
  e^{(2x-1)^2+(2y-1)^2}, & \text{if} \ (x,y)\in\Omega_1,\\
  0.1(x^2+y^2)^2-0.01{\rm ln}(2\sqrt{x^2+y^2}), & \text{if} \ (x,y)\in\Omega_2.
  \end{array}
  \right.
 \]

Plotted in Figure \ref{fig:9} is the interface and the domain (left),
the initial mesh (middle), and the next level mesh by the refinement of the initial mesh (right).
Figure \ref{fig:10} shows the surface plot of the numerical solution $u_h$ calculated by the PDWG method with $k=1$. Table \ref{6} reports the approximation error and the corresponding rate of convergence for $u_h$ and $\lambda_h$. An optimal order of convergence of $\mathcal{O}(h^{k})$ for $\|u-u_h\|_0$ and $\|\lambda_h\|_w$ is observed, which is in good consistency with our theoretical
findings in Theorem \ref{theoestimate}. Table \ref{6} further suggests a convergence for $\|\lambda_0\|_1$ at the rate of $\mathcal{O}(h^{k+1})$, and a convergence for $\|\lambda_0\|_0$ at the rates of $\mathcal{O}(h^{k+1})$ and $\mathcal{O}(h^{k+2})$ for the linear and quadratic PDWG methods, respectively.

 \begin{figure}[htb]
  \centering
   \includegraphics[width=.32\textwidth]{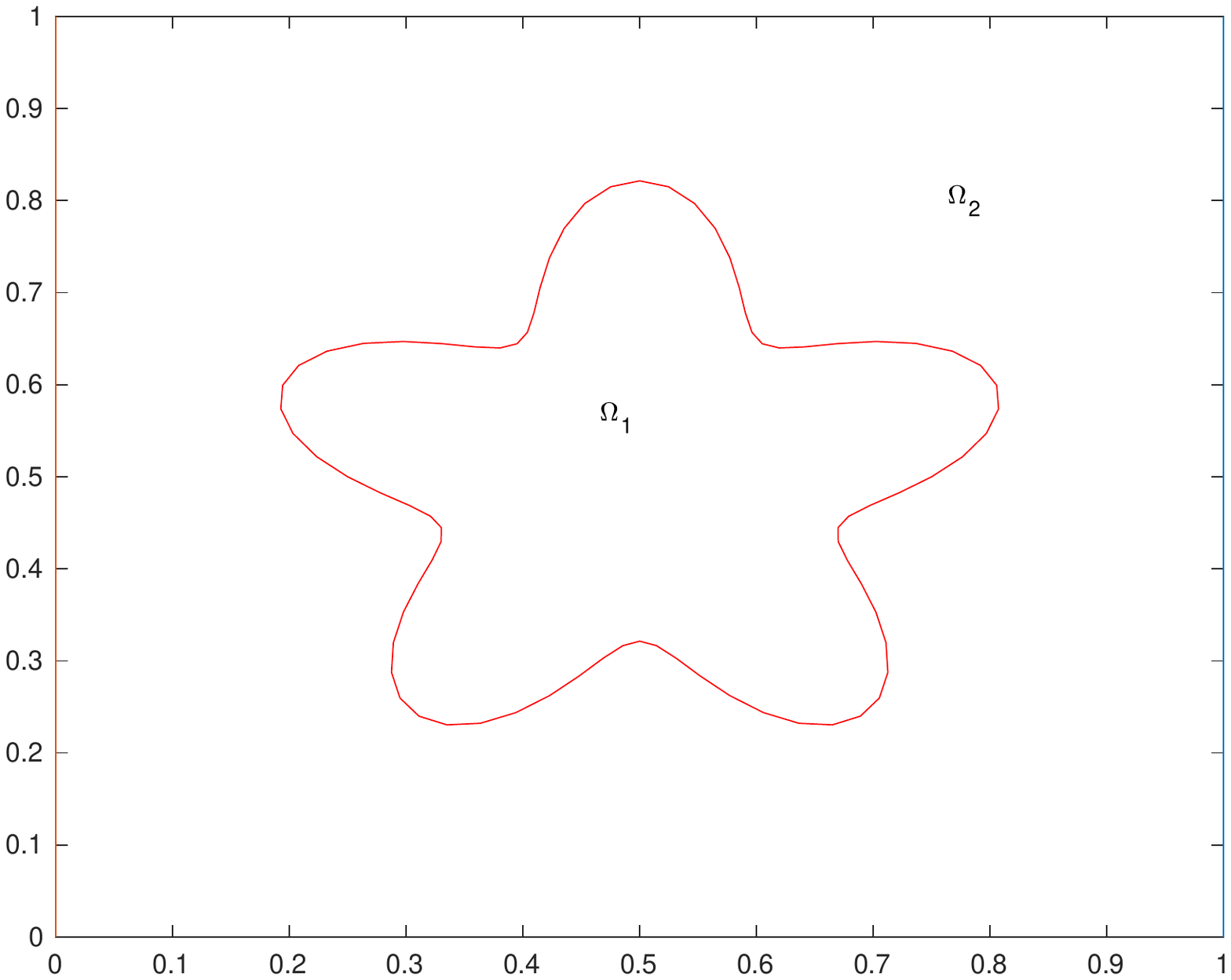}
  \includegraphics[width=.32\textwidth]{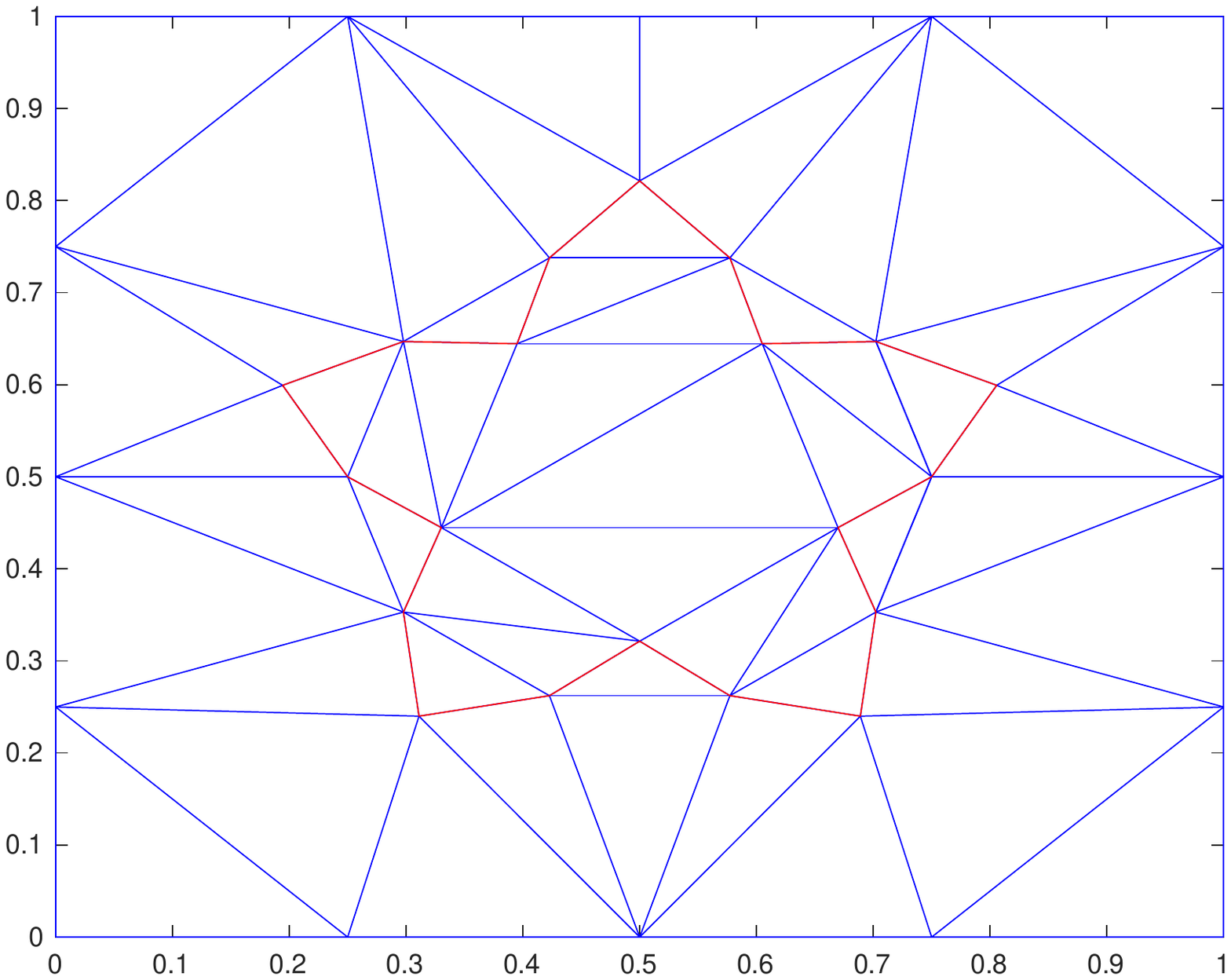}~
  \includegraphics[width=.32\textwidth]{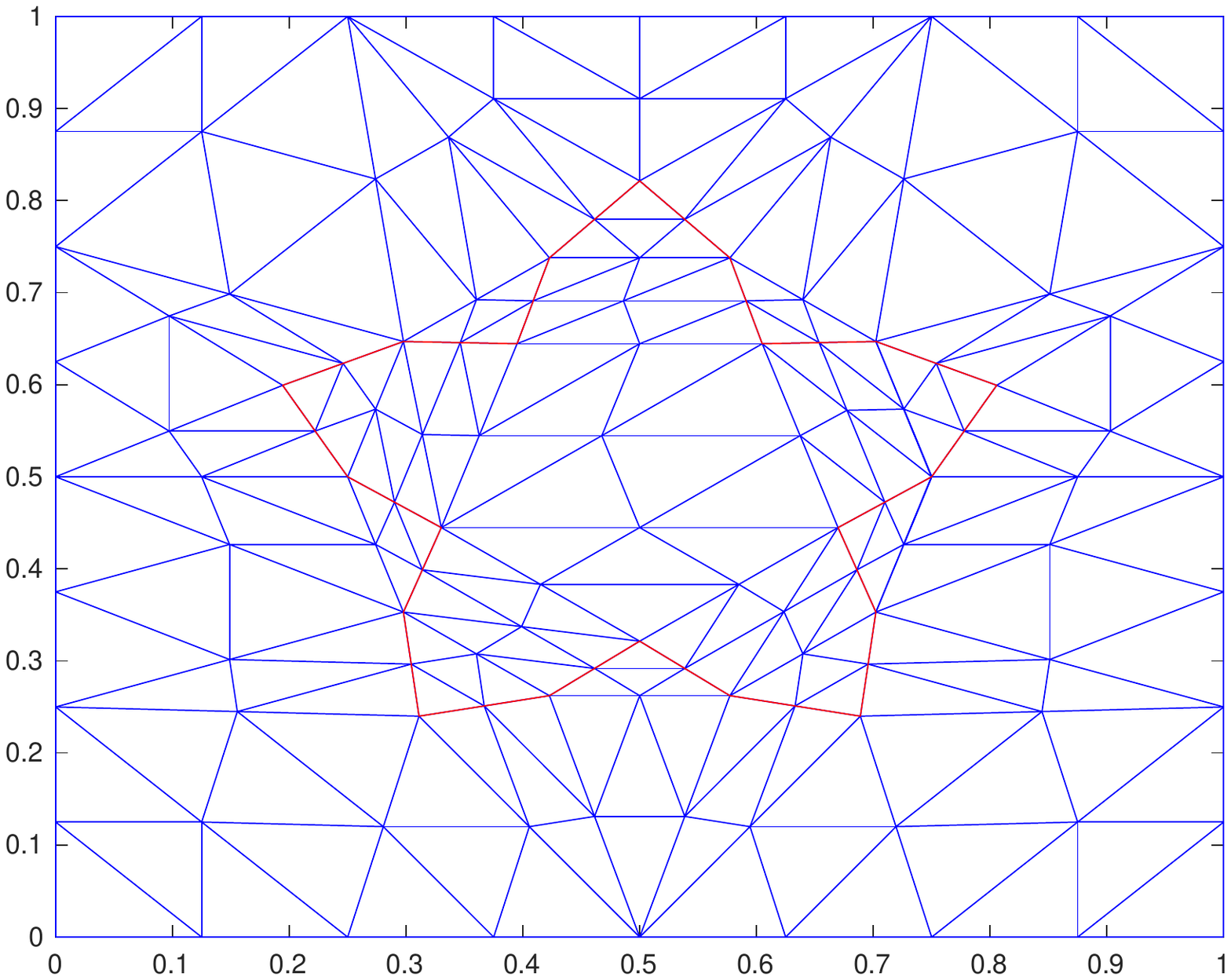}
\caption{The interface and subdomains (left) and the initial mesh (middle) and the next level of mesh after one refinement (right).}
  \label{fig:9}
\end{figure}

\begin{figure}[htb]
  \centering
  \includegraphics[width=.6\textwidth]{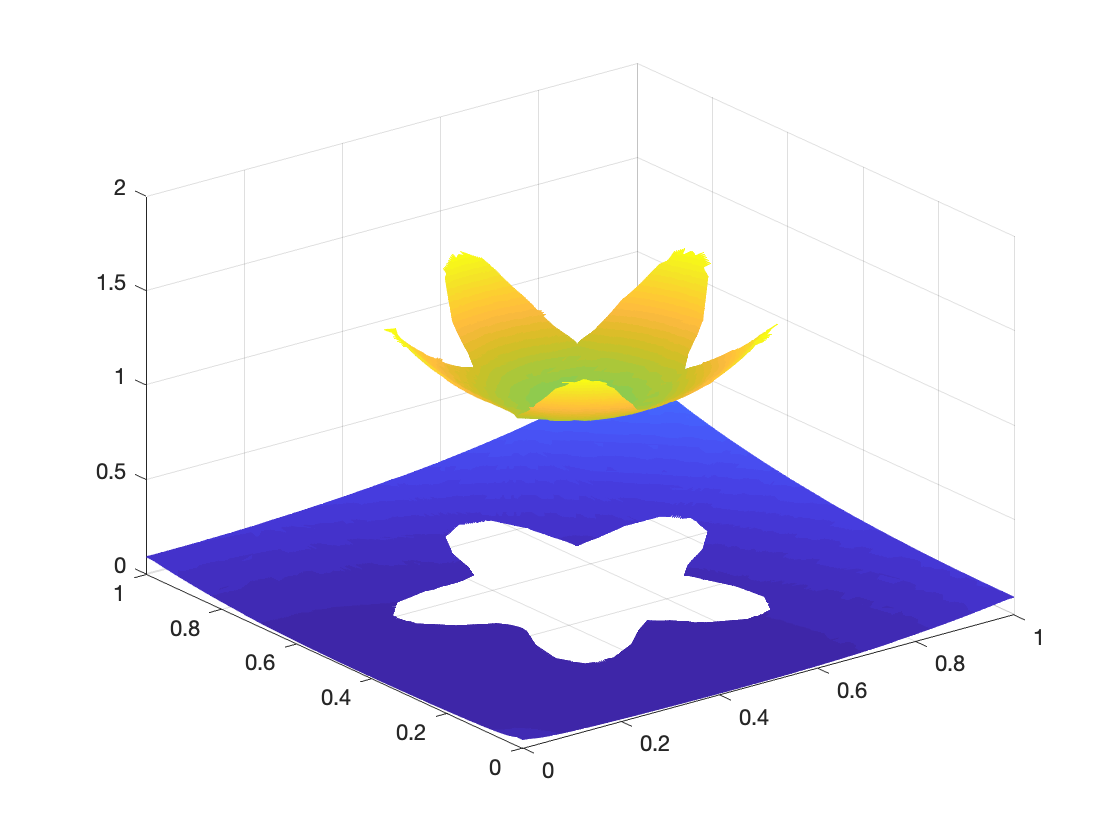}~
 \caption{ Surface plot of the numerical solution $u_h$ calculated by the PDWG method of Example 5.}
  \label{fig:10}
\end{figure}

\begin{table}[htbp]\caption{Errors and  convergence rates of the linear and quadratic PDWG methods for Example 3.}\label{6}\centering
\begin{threeparttable}
        \begin{tabular}{c |c |c | c | c | c | c | c | c | c| }
        \hline
  & $h$ &   $\|\lambda_h\|_w$ & rate& $\|\lambda_0\|_1$  & rate & $\|\lambda_0\|_0$ & rate & $\|u-u_h\|_0$ & rate \\
       \hline \cline{2-10}
        &1.88e-1&   4.77e-2  &   -- &  1.20e-1 &     -- &  1.13e-2  &   --&   7.18e-2 &          --\\
        & 9.38e-2 &  2.74e-2 &  0.80 &  5.66e-2 &  1.08 &  5.07e-3  & 1.16&   3.35e-2 &  1.10\\
 $k=1$  & 4.69e-2 &  1.54e-2 &  0.83 &  1.70e-2 &  1.74 &  1.55e-3 &  1.71&   2.01e-2 &  0.74\\
        & 2.34e-2&   8.33e-3 &  0.89 &  4.85e-3 &  1.80 &  4.40e-4 &  1.81 &  9.26e-3 &  1.12\\
        & 1.17e-2 &  4.26e-3 &  0.97 &  1.27e-3 &  1.94&  1.15e-4 &  1.94 &  4.77e-3 &  0.96\\
\hline
& $h$ &   $\|\lambda_h\|_w$ & rate& $\|\lambda_0\|_1$  & rate & $\|\lambda_0\|_0$ & rate & $\|u-u_h\|_0$ & rate \\
       \hline \cline{2-10}
       &  3.8139e-1&    1.69e-2 &       -- &   7.29e-2 &     -- &   5.63e-3 &      -- &   1.12e-2  &     --\\
       &  1.9069e-1 &   6.38e-3 &   1.41&   2.47e-2&    1.56 &   8.29e-4 &   2.76 &   2.97e-3  &  1.91\\
 $k=2$ &  9.5347e-2 &   1.79e-3 &   1.83 &   3.73e-3  &  2.73 &   6.17e-5 &   3.75 &   8.28e-4  &  1.85\\
       &  4.7674e-2 &   5.17e-4 &   1.80 &   5.17e-4 &   2.85 &   4.36e-6 &   3.82 & 2.33e-4  &  1.83\\
       &  2.3837e-2 &   1.58e-4 &   1.71 &   6.91e-5 &   2.90 &   2.92e-7 &   3.90  &  5.42e-5 &   2.11\\
\hline
 \end{tabular}
 \end{threeparttable}
\end{table}

{\it Example 6:} We consider the problem \eqref{model}-\eqref{model2} on the domain $\Omega=(0,1)^2$ with the same interface as that of Example 1; i.e.,
$\Omega_1=(0.25,0.75)^2$, $\Omega_2=\Omega\backslash\Omega_1$, and $\Gamma=\partial\Omega_1$. The coefficients are set as
\[
     a_{1}=2+\sin(x+y), \ \ a_{2}=5, \ \ {\bf b}_1= {\bf b}_2=(0,0),\ \ c_1=c_2=0.4.
\]
 Define
\begin{eqnarray*}
  && \Gamma_1=\{(\frac 14, y): \frac 14\le y\le \frac 34\},\ \ \Gamma_2=\{(\frac 34, y): \frac 14\le y\le \frac 34\},\\
  && \Gamma_3=\{(x,\frac 14): \frac 14\le x\le \frac 34\},\ \ \Gamma_4=\{(x,\frac 34): \frac 14\le x\le \frac 34\}.
\end{eqnarray*}
 We choose the boundary condition $u|_{\partial\Omega}=\frac{1}{5}\sin(x+y)+\cos(x+y)+1$, and
 the following interface data:
\begin{eqnarray*}
  &&  [\![u]\!]_{\Gamma_i}=i,\ \ 1\le i\le 4,\ \  [\![a\nabla u-{\bf b}u]\!]_{\Gamma_1}=(4,0),\ \ [\![a\nabla u-{\bf b}u]\!]_{\Gamma_2}=(2/e^{\frac 34}e^x,0),\\
  &&
  [\![a\nabla u-{\bf b}u]\!]_{\Gamma_3}=(0,6\pi\cos(2\pi y)),\ \ \ [\![a\nabla u-{\bf b}u]\!]_{\Gamma_4}=(1,0).
\end{eqnarray*}

Figure \ref{fig:11} show the plots for the numerical solution $\lambda_h$ (left) and $u_h$ (right) obtained from the PDWG numerical method with $k=1$ for the interface problem when the right-hand side functions are taken as $f_1=f_2=0$. It should be noted that the exact solution to this interface problem is not known, and the interface data for the jump of $u$ is piecewise constant and, therefore, does not have the $H^{\frac12}(\Gamma)$-regularity needed in most other numerical methods.

\begin{figure}[htb]
  \centering
  \includegraphics[width=.42\textwidth]{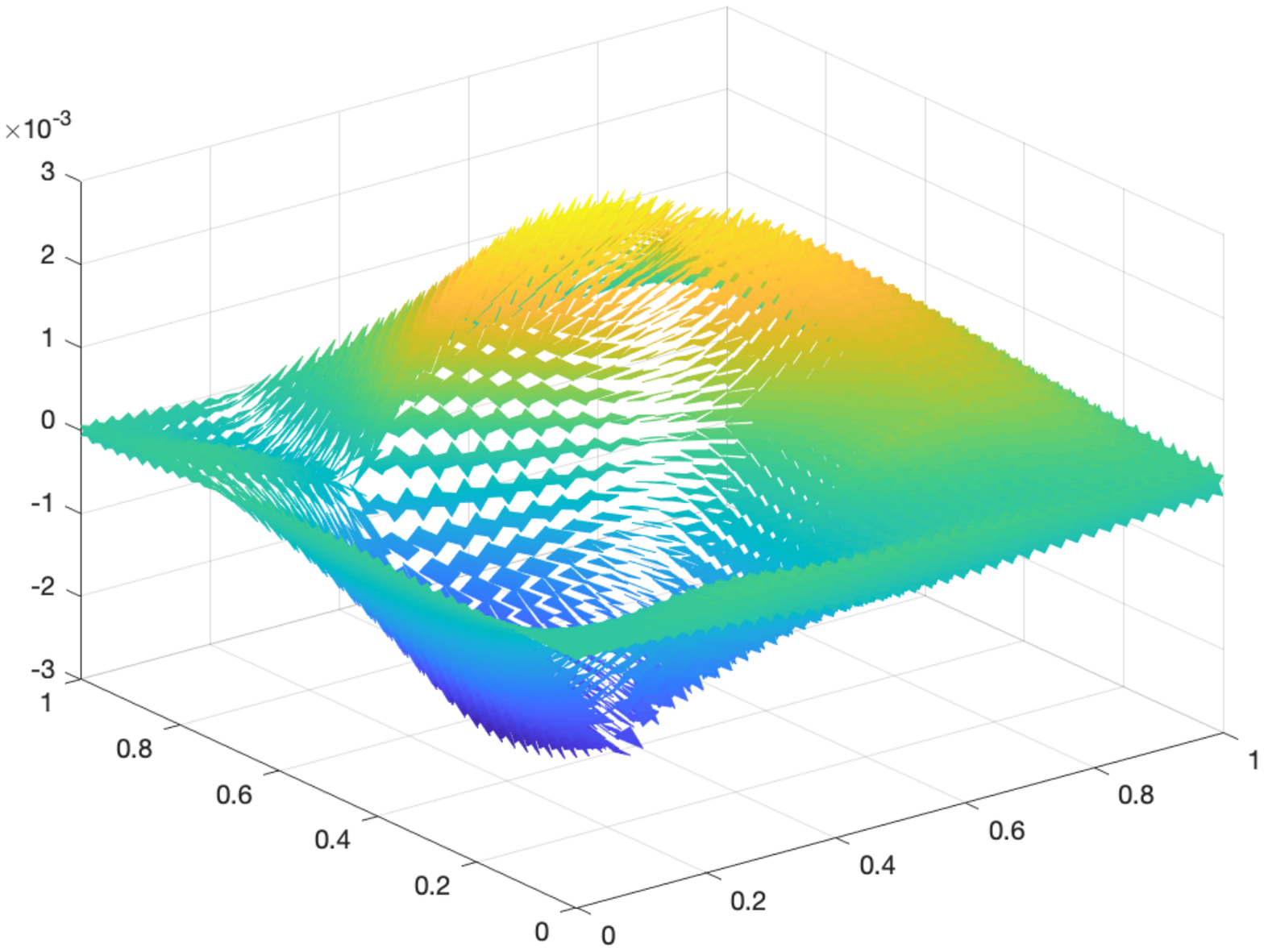}
  \includegraphics[width=.42\textwidth]{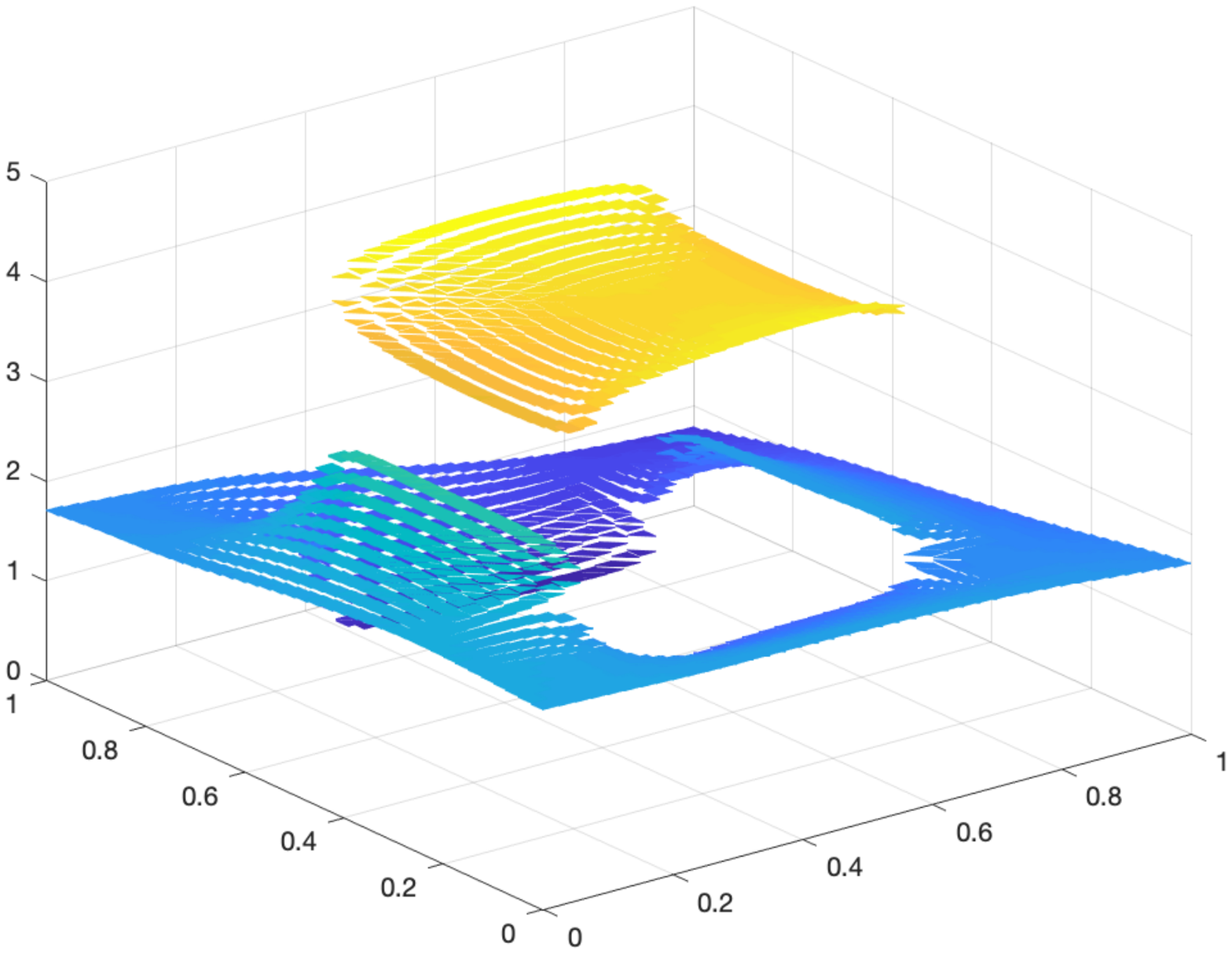}~
 \caption{Surface plot of the numerical solution $\lambda_h$ (left) and $u_h$ (right) calculated by the PDWG method of Example 6 with $f_1=f_2=0, k=1$ and $h=1/40$.}
  \label{fig:11}
\end{figure}
%

%

{\it Example 7:} This example assumes the same interface $\Gamma$ as in Example 6. Here, we take
\[
     a_{1}=1, \ \ a_{2}=100, \ \ {\bf b}_1= {\bf b}_2=(2+y,1+x),\ \ c_1=c_2=0, \ f_1=f_2=0.
\]
The boundary condition and the interface data are chosen as
\begin{eqnarray*}
  &&u|_{\partial\Omega}=\frac 12(x^2+y^3)(\frac 12\sin(x+y)+\frac 13\cos(x+y))-\frac 13 {\rm ln} (x^2+y^2),\ \\
  &&[\![u]\!]_{\Gamma_i}=1, [\![u]\!]_{\Gamma_{i+2}}=0,\  i=1,2,\\
  && [\![a\nabla u-{\bf b}u]\!]_{\Gamma_1}=(\pi\cos(2\pi x),0),\ \  [\![a\nabla u-{\bf b}u]\!]_{\Gamma_2}=(\frac12\sin(x)+\frac 14\cos(x)+y,0),\\
  && [\![a\nabla u-{\bf b}u]\!]_{\Gamma_j}=((y-\frac 14)(y-\frac 34)(\cos(x)+2x),(\sin(x)+x^2)(2y-1)),\ \ \ j=3,4.
\end{eqnarray*}

Figure \eqref{fig:14} shows the plots for the numerical solution $\lambda_h$ (left) and $u_h$ (right) obtained from the PDWG numerical method for the interface problem with $k=2$. For this test case, the exact solution to the interface problem is not known. Furthermore, the interface data for the jump of $u$ is discontinuous by assuming $0$ or $1$ so that the $H^{\frac12}(\Gamma)$-regularity is not satisfied. The PDWG method, however, is applicable and provide meaningful numerical solutions.

%

\begin{figure}[htb]
  \centering
  \includegraphics[width=.42\textwidth]{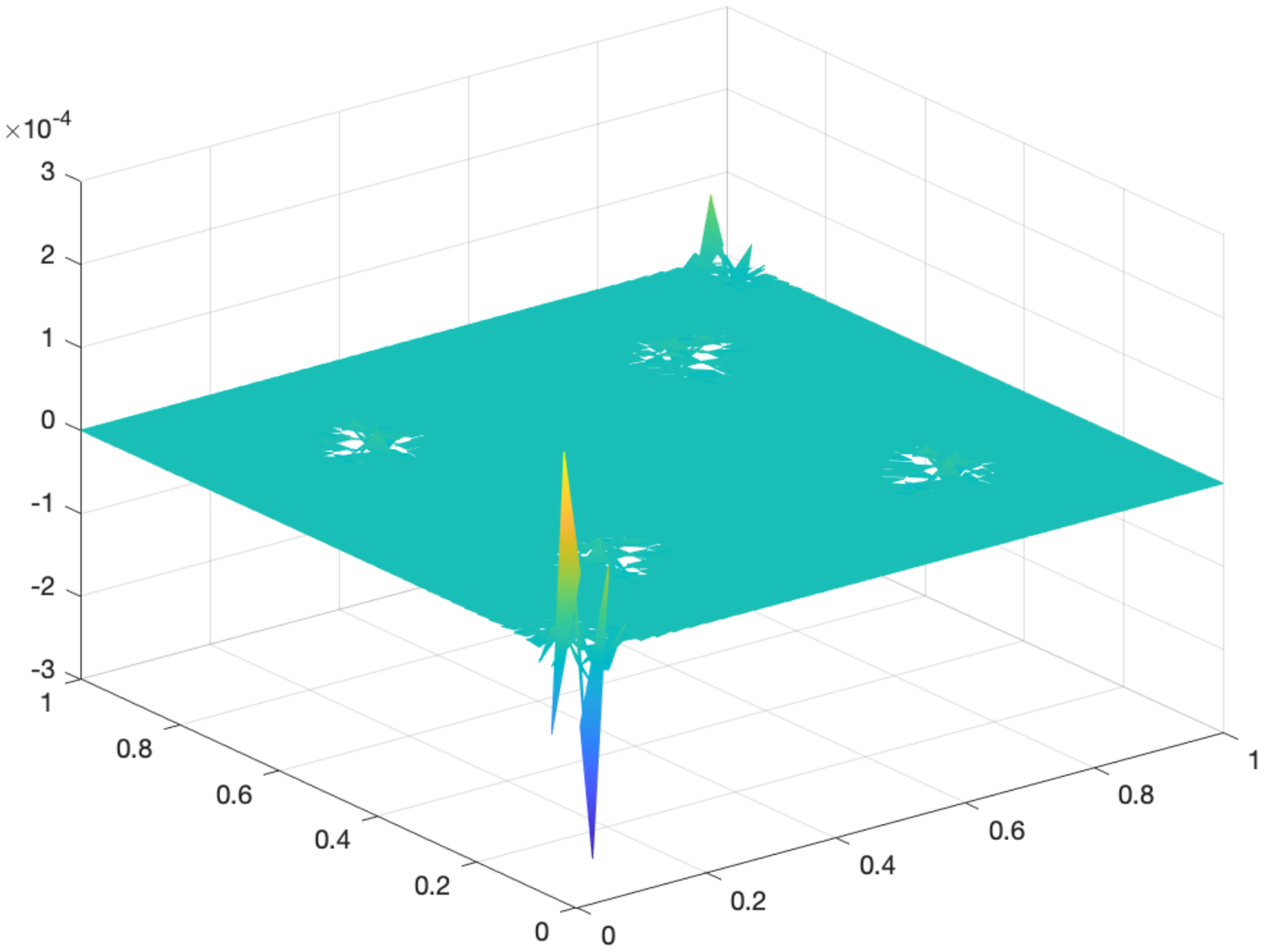}
  \includegraphics[width=.42\textwidth]{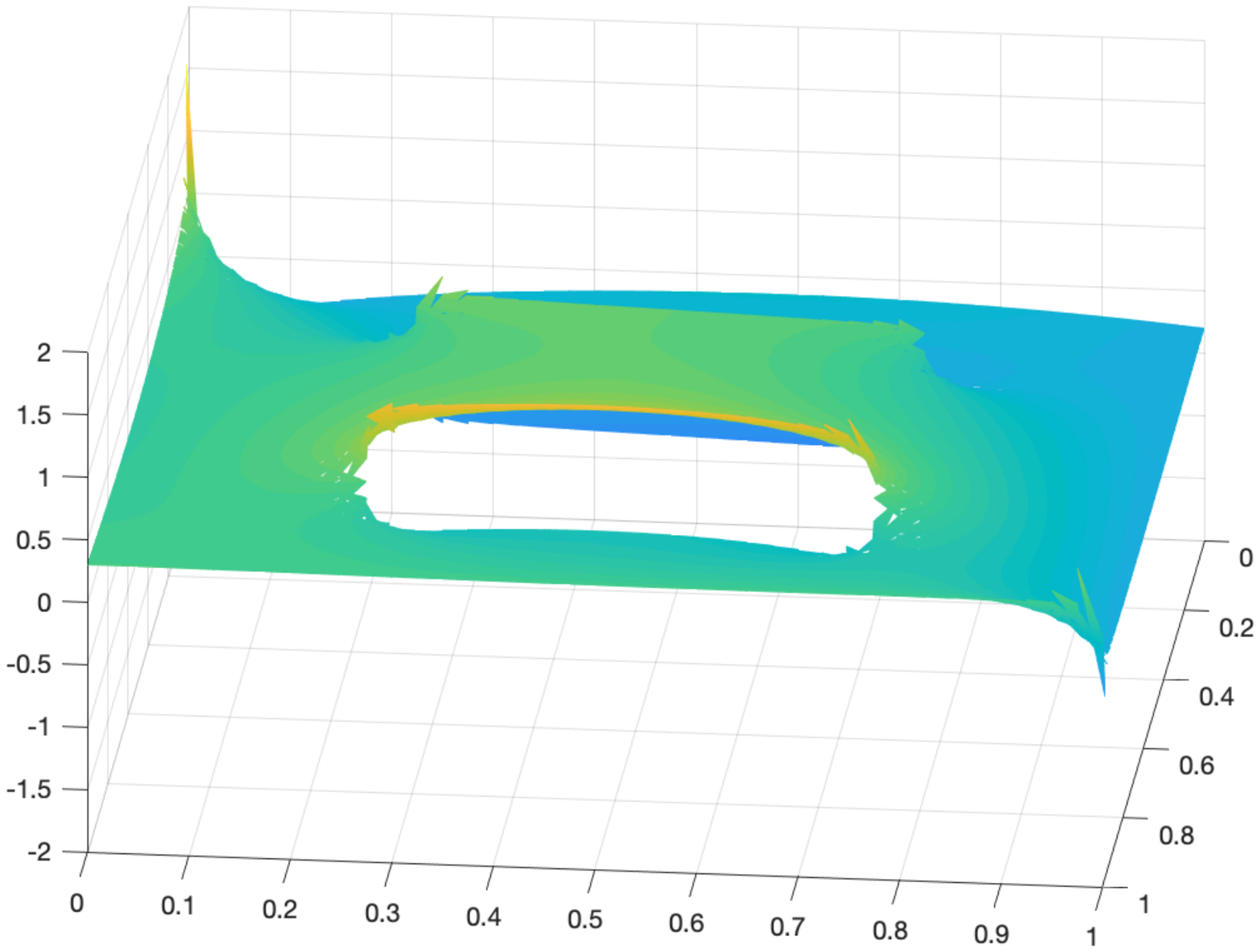}~
 \caption{Surface plot of the numerical solution $\lambda_h$ (left) and $u_h$ (right) calculated by the PDWG method of Example 7 with $f_1=f_2=0,k=2$ and $h=1/40$.}
  \label{fig:14}
\end{figure}

\bibliographystyle{abbrv}
\bibliography{Ref}


\end{document}